\documentclass[12pt]{amsart}
\usepackage{amsfonts}
\numberwithin{equation}{section}
\usepackage{amssymb}
\usepackage{amsmath}
\def\arg{\operatorname{arg}}

\newtheorem{lem}{Lemma}[section]
\newtheorem{thm}{Theorem}[section]

\newtheorem{cor}{Corollary}[section]
\theoremstyle{definition}
\newtheorem{defi}{Definition}[section]
\theoremstyle{remark}
\newtheorem*{re}{Remark}

\title[Holomorphic Curves]{Propagation Sets of Holomorphic Curves}
\author[Zheng and Yan]{ Zheng Jian-Hua and Yan Qi-Ming}
\thanks{{2000 Mathematics Subject Classification: Primary 32H30, Secondary 30D35}\\
{\mbox{} \hspace{0.56em} Funding: This work was supported by the National Natural Science Foundation of China
[grant numbers 11571193, 11571256].
}}
\address{Department of Mathematical Sciences, Tsinghua University, Beijing, 100084, P. R. China}
\email{zheng-jh@mail.tsinghua.edu.cn}
\address{School of Mathematical Sciences, Tongji
University,  Shanghai 200092, P. R. China}
\email{yan$\underline{\mbox{ }}$qiming@hotmail.com }

\keywords{Holomorphic Curves, Uniqueness
Problem, Propagation Set}
\begin{document}

\begin{abstract}
We consider a problem of whether a property of holomorphic curves on
a subset $X$ of the complex plane can be extended to the whole
complex plane. In this paper, the property we consider is uniqueness
of holomorphic curves. We introduce the propagation set. Simply
speaking, $X$ is a propagation set if linear relation of holomorphic
curves on the part of preimage of hyperplanes contained in $X$ can
be extended to the whole complex plane. If the holomorphic curves
are of infinite order, we prove the existence of a propagation set
which is the union of a sequence of disks (In fact, the method
applies to the case of finite order). For a general case, the union
of a sequence of annuli will be a propagation set. The classic
five-value theorem and four-value theorem of R. Nevanlinna are
established in such propagation sets.
\end{abstract}
%\maketitle %\keywords{Keywords and phases: Nevanlinna Theory,
%Holomorphic Curves, Angular Domain}

\maketitle

\section{Introduction and Results}\label{intr}

Let $X$ be a subset of the complex plane $\mathbb{C}$. Let $f$ and
$g$ be two meromorphic functions on $\mathbb{C}$. We say that $f$
and $g$ share value $a$ IM on $X$ if $f^{-1}(a)\cap X=g^{-1}(a)\cap
X$, that is, $f(z)=a$ on $X$ if and only if $g(z)=a$ on $X$; $f$ and
$g$ share value $a$ CM on $X$ if $f^{-1}(a)\cap X=g^{-1}(a)\cap X$
counted according to multiplicities. R. Nevanlinna proved the
five-value theorem that if $f$ and $g$ share five distinct values IM
on $\mathbb{C}$, then $f\equiv g$. We consider the problem of
whether $\mathbb{C}$ in the five-value theorem is replaced by a
precise subset $X$ of $\mathbb{C}$.

Therefore, we introduce a concept: A subset $X$ of $\mathbb{C}$ is
said to be a five-value uniqueness set of two meromorphic functions
$f$ and $g$ if it is true that $f\equiv g$ as long as $f$ and $g$
share five distinct values IM on $X$.

Let $f$ be a meromorphic function. If it is transcendental, the
Picard Theorem says that $f$ can take infinitely often all but at
most two of values on the extended complex plane $\hat{\mathbb{C}}$.
It is improved by the Borel Theorem that the convergence exponent of
$a$-points of $f$ in $\mathbb{C}$, with exception of at most two
values $a$ of $\hat{\mathbb{C}}$, is equal to the growth order of
$f$. Certainly, it makes sense only when the growth order of $f$ is
positive or $\infty.$ The result of Borel Theorem also holds even on
a sequence of disks instead of $\mathbb{C}$, which are so-called
filling disks of $f$. That is, the convergence exponent of
$a$-points of $f$ on the sequence of disks equals to the growth
order of $f$ with exception of at most two values $a$ of
$\hat{\mathbb{C}}$ (cf. \cite{Yang}). In \cite{Zheng2}, we give a
definition of filling disks of holomorphic curves and prove their
existence.

Motivated by the filling disks, we ask if there exists a sequence of
disjoint disks whose union is a five-value uniqueness set. The main
purpose of this paper is to study the problem for holomorphic
curves, while a meromorphic function is considered as a holomorphic
curve. We will give conditions which confirm the existence of such a
disk sequence.

Let $\mathcal{P}^n(\mathbb{C})$ be the $n$-dimensional complex
projective space, that is,
$\mathcal{P}^n(\mathbb{C})=\mathbb{C}^{n+1}\setminus\{0\}/\sim,$
where $\sim$ is the equivalent relation defined by that
$(a_0,a_1,...,a_n)\sim (b_0,b_1,...,b_n)$ if and only if
$(a_0,a_1,...,a_n)=\lambda(b_0,b_1,...,b_n)$ for some
$\lambda\in\mathbb{C}\setminus\{0\}.$ We write $[a_0:a_1:...:a_n]$
for the equivalent class of $(a_0,a_1,...,a_n)$.

A map $f:\mathbb{C}\rightarrow \mathcal{P}^n(\mathbb{C})$ is called
a holomorphic curve on the complex plane $\mathbb{C}$, if we can
write $f=[f_0:f_1:...:f_n]$ where every $f_j$ is an entire function
and they have no common zeros on $\mathbb{C}$ and
${\bf{f}}=(f_0,f_1,...,f_n)$ is called a reduced representation of
$f$. In particular, for $n=1$, $f=[f_0:f_1]$ is a meromorphic function, and denote by $\nu_f$
the divisor defined by the zeros of entire function $f_0$ and $\nu^{\infty}_f$ the divisor
defined by the zeros of entire function $f_1$.

Let $f$ be a holomorphic curve on $\mathbb{C}$ with a reduced
representation ${\bf f}=(f_0,f_1,...,f_n)$. Set
$$v_f(z)=\bigvee_{j=0}^n\log |f_j(z)|.$$
The Cartan characteristic of $f$ is
$$T(r,f)=\frac{1}{2\pi}\int_0^{2\pi}v_f(re^{i\theta}){\rm d}\theta- v_f(0).$$
Since $v_f(z)$ is subharmonic, $T(r,f)$ is a positive logarithmic
convex increasing real-value function. If for some $f_i\not\equiv
0$, at least one of $f_j/f_i (j=0,1,2,...,n)$ is transcendental,
then $T(r,f)/\log r\to\infty (r\to\infty)$.

The order and lower order of a positive non-decreasing real-value
function $T(r)$ are respectively defined by
$$\lambda=\limsup\limits_{r\to\infty}\frac{\log T(r)}{\log r},\
\tau=\liminf\limits_{r\to\infty}\frac{\log T(r)}{\log r}.$$ Then the
order and lower order of a holomorphic curve $f$ on the complex
plane are respectively the order and lower order of $T(r,f)$,
denoted by $\lambda(f)$ and $\tau(f)$.

A hyperplane $H$ in $\mathcal{P}^n(\mathbb{C})$ is
$$H=\left\{[x_0:x_1:...:x_n]:\sum_{k=0}^na_kx_k=0\right\},$$
where $(a_0,a_1,...,a_n)\in\mathbb{C}^{n+1}\setminus\{0\}$.
Obviously, $H$ is completely determined by $[a_0:a_1:...:a_n]$.
Sometimes, we call the non-zero vector ${\bf{a}}=(a_0,a_1,...,a_n)$
as a hyperplane. And for a hyperplane ${\bf a}=(a_0,a_1,...,a_n)$,
write $\langle{\bf f,a}\rangle:=a_0f_0+a_1f_1+...+a_nf_n$ and we always assume
that $\langle{\bf f, a}\rangle\not\equiv 0$ when it appears. The divisor $\nu_{\langle{\bf f, a}\rangle}$ can
be regarded as a map from $\mathbb{C}$ to $\mathbb{Z}$ such that, for each $z\in \mathbb{C}$,
 $\nu_{\langle{\bf f, a}\rangle}(z)$ is the intersection multiplicity of the images of
$f$ and $H_j$ at $f(z)$.

Let ${\bf a}_1,{\bf a}_2,...,{\bf a}_q$ be $q$ hyperplanes. They are called to be
in general position if every $n+1$ members of ${\bf a}_1,{\bf a}_2,...,{\bf a}_q$
are linearly independent, that is, if no $n+1$ members of
${\bf a}_1,{\bf a}_2,...,{\bf a}_q$ are linearly dependent. A holomorphic curve
is linearly non-degenerated if its image cannot be contained in a
hyperplane.

We introduce the following concept.

\begin{defi}\ A subset $X$ of $\mathbb{C}$ is called
$(p,s;d)$-propagation set of $q$ holomorphic curves $f_i$ with a
reduced representation ${\bf f}_i\ (i=1,2,...,q)$ with $2\leq s\leq
q$ and $d\ge 1$, if the following statement holds:\ given any $p$ hyperplanes
$H_j (j=1,2,...,p)$ in general position satisfying
$$\min\{\nu_{\langle{{\bf f}_1, {\bf a}_j}\rangle}(z),d\}=
\min\{\nu_{\langle{{\bf f}_2, {\bf a}_j}\rangle}(z),d\}=...=
\min\{\nu_{\langle{{\bf f}_q, {\bf a}_j}\rangle}(z),d\},\ \forall z\in X,$$
for $j=1,2,...,p$ and $f_1^{-1}(H_i\cap H_j)\cap X=\emptyset$ for $i\not=j$, if for any $1\leq
j_1<j_2<...<j_s\leq q$,
\begin{equation}\label{1.1}{\bf f}_{j_1}(z)\wedge{\bf
f}_{j_2}(z)\wedge\cdots\wedge{\bf f}_{j_s}(z)=0,\ \forall\ z\in
\bigcup_{j=1}^pf_1^{-1}(H_j)\cap X,\end{equation} then we have
\begin{equation}\label{11}{\bf f}_{1}(z)\wedge{\bf
f}_{2}(z)\wedge\cdots\wedge{\bf f}_{q}(z)\equiv 0,\ \forall\ z\in
\mathbb{C},\end{equation} where $\wedge$ means the exterior product.
\end{defi}

(For simplicity, we write $(p,s)$-propagation set for $(p,s;1)$-propagation set. In this case,
$\min\{\nu_{\langle{{\bf f}_1, {\bf a}_j}\rangle}(z),d\}=
\min\{\nu_{\langle{{\bf f}_2, {\bf a}_j}\rangle}(z),d\}=...=
\min\{\nu_{\langle{{\bf f}_q, {\bf a}_j}\rangle}(z),d\}$ on $X$ is equivalent to $
f^{-1}_1(H_j)\cap X=f^{-1}_2(H_j)\cap X=...=f^{-1}_q(H_j)\cap X$.)

Since $\mathcal{P}^1(\mathbb{C})$ is the Riemann sphere, a
holomorphic curve $f: \mathbb{C}\to \mathcal{P}^1(\mathbb{C})$ is a
meromorphic function and a hyperplane is a value on
$\hat{\mathbb{C}}$. Let $f$ and $g$ be two meromorphic functions.
Then $A:=f^{-1}(a)\cap X=g^{-1}(a)\cap X$ is ${\bf f}\wedge {\bf
g}=0, \forall\ z\in A.$ A subset $X$ of $\mathbb{C}$ is a five-value
uniqueness set of two meromorphic functions if and only if it is a
$(5,2)$-propagation set of them. This is because ${\bf f}\wedge {\bf
g}=0,\ z\in\mathbb{C},$ is equivalent to that $f\equiv g$ and that
hyperplanes are in general position is equivalent to that the
corresponding values are distinct.

In 1989, Stoll \cite{Stoll} proved the following, which is stated in
view of propagation set for our convenience.

\

\noindent {\sl {\bf Theorem A.}\ $\mathbb{C}$ is a
$(p,s)$-propagation set of $q$ linearly non-degenerate holomorphic
curves, if $p>\frac{qn}{q-s+1}+n+1$.}

\

Note that $p>\frac{qn}{q-s+1}+n+1$ is equivalent to that
$q>(s-1)\frac{p-n-1}{p-2n-1}$ and $p>2n+1$. From Stoll's Theorem
{\bf A} with $q=s=2$ and $p=3n+2$, we naturally deduce the result:
{\sl Let $f$ and $g$ be two linearly non-degenerate holomorphic
curves and let $\{H_j\}_{j=1}^p$ be $p(>3n+1)$ hyperplanes in
general position. Assume that

i)\ $f^{-1}(H_j)=g^{-1}(H_j), j=1,2,...,p;$

ii)\ $f=g$ at $z\in\bigcup_{j=1}^pf^{-1}(H_j)$.

\noindent Then $f\equiv g$ on $\mathbb{C}$.}

This result is due to Smiley \cite{Smiley}. However,  the result is
proved in \cite{ChenYan} for $q=s=2$, $p=2n+3$ and $d=1$ and in \cite{Si2}
for $q=s=2$, $p=2n+2$ and $d=n+1$ when $n\geq 2$, which is not included in the Stoll's
Theorem {\bf A}. They are of independent significance and the best
results at present. These results are the improvements of the
Fujimoto's results in \cite{Fujimoto}, \cite{Fujimoto2} and
\cite{Fujimoto3}. Uniqueness problem of holomorphic curves attracts
many interests. For references, we list some of the papers about
this topic such as \cite{Ji}, \cite{Fujimoto1}, \cite{Ru1},
\cite{ChenRu}, \cite{DethloffTan} and \cite{DulockRu}.

In this paper, we discuss the possibility of that a precise subset
of $\mathbb{C}$ is a propagation set.

Firstly, we consider the uniqueness problem ignoring multiplicities.

By $B(z,r)$ we denote the disk
centered at $z$ with radius $r$ and by $B(r)$ the disk centered at
the origin with radius $r$. For a sequence of complex numbers
$\{z_m\}$ and a sequence of positive numbers $\{\varepsilon_m\}$,
set $B(\{z_m\};\{\varepsilon_m\}):=\bigcup_{m=1}^\infty
B(z_m,\varepsilon_m|z_m|).$

We establish the following.

\begin{thm}\label{thm1.1}\ Let $f_i(z)\ (i=1,2,...,q; q\geq 2)$ be linearly
non-degenerate transcendental holomorphic curves such that
$T(r):=\sum_{i=1}^qT(r,f_i)$ is of infinite order. Let $\{\varepsilon_m\}$
be a sequence of positive numbers such that $\varepsilon_m\to 0$ as
$m\to\infty$. Then there exists a sequence of complex numbers
$\{z_m\}$ with $z_m\to\infty$ as $m\to\infty$ such that
$B(\{z_m\};\{\varepsilon_m\})$ is a $(p,s)$-propagation set of
$f_i(z)\ (i=1,2,...,q)$, if $p>\frac{qn}{q-s+1}+n+1$.
\end{thm}

For $q=s=2$ and $p=2n+3$, Theorem \ref{thm1.1} can be improved as follows
which corresponds to the result of \cite{ChenYan}.

\begin{thm}\label{thm1.4}\ Let $f$ and $g$ be two linearly
non-degenerate transcendental holomorphic curves such that $T(r):=T(r,f)+T(r,g)$
 is of infinite order. Let $\{\varepsilon_m\}$ be
a sequence of positive numbers such that $\varepsilon_m\to 0$ as
$m\to\infty$. Then there exists a sequence of complex numbers
$\{z_m\}$ with $z_m\to\infty$ as $m\to\infty$ such that
$B(\{z_m\};\{\varepsilon_m\})$ is a $(2n+3,2)$-propagation set of
$f$ and $g$.
\end{thm}

\begin{re}
(i) Theorems \ref{thm1.1} and \ref{thm1.4} are significant results in which the
propagation set can be a precise subset of $\mathbb{C}$. In fact, we
can choose $\{z_m\}$ such that $B(z_m,\varepsilon_m|z_m|)$ are
disjoint each other. When $n=1$, it shows that Nevanlinna's
five-value theorem is valid if $f$ and $g$ share five values on a
sequence of disjoint disks in $\mathbb{C}$ and if at least one of
$f$ and $g$ is of infinite order. In the proof of these two theorems (see Section 3 and Section 4), the value distribution of holomorphic
curves on angular domains established by the first author is
essentially used.

(ii) In these two theorems, $T(r)$ is assumed to
be of infinite order, i.e., at least one of $f_i (i=1,2,...,q)$ is
of infinite order. If $T(r)$ is of finite order with
the order $0<\lambda<+\infty$, in view of our method we can obtain
similar conclusion for $\{\varepsilon_m\}$ with
$\varepsilon_m>\frac{\pi}{2\lambda}$ and
$\varepsilon_m\to\frac{\pi}{2\lambda}(m\to\infty).$ But we leave the
proof for the finite order case to the reader. Here we just mention
that for the finite order case, we use the P\'olya peak sequence
(its definition and existence will be stated in Section 2) instead
of Lemma 2.4 to find the points $\{z_m\}$.
\end{re}

By $A(r,R)$ we denote the annulus of $\{z:\ r<|z|<R\}$. For a
sequence of positive numbers $\{r_m\}$ with $r_m>1$ and $\sigma>1$,
set $A(\{r_m\};\sigma):=\bigcup_{m=1}^\infty A(r_m,r_m^\sigma).$

\begin{thm}\label{thm1.2}\ Let $\{r_m\}$ be a
sequence of positive numbers tending to $\infty$. Then
$A(\{r_m\};\sigma)$ is a $(p,s)$-propagation set of any $q$ linearly
non-degenerate transcendental holomorphic curves $f_i(z)\
(i=1,2,...,q)$ if
$\left(1-\frac{1}{\sigma}\right)p>\frac{qn}{q-s+1}+n+1$.\end{thm}

For $q=s=2$ and $p=2n+3$, we have the following.

\begin{thm}\label{thm1.6}\ Let $\{r_m\}$ be a
sequence of positive numbers tending to $\infty$. Then
$A(\{r_m\};\sigma)$ is a $(2n+3,2)$-propagation set of any two linearly
non-degenerate transcendental holomorphic curves $f$ and $g$ if
${\sigma}>\frac{5n(2n+3)}{3n+2}$.
\end{thm}

In particular, from Theorem \ref{thm1.2} or Theorem \ref{thm1.6} we deduce that
$A(\{r_m\};6)$ is a five-value uniqueness set of two meromorphic
functions. In fact, for $n=1, q=s=2, p=5$ and $\sigma=6$, we have
$$\left(1-\frac{1}{\sigma}\right)p=\frac{25}{6}>4=\frac{qn}{q-s+1}+n+1.$$
Then $A(\{r_m\};6)$ is a $(5,2)$-propagation set so that it is a
five-value uniqueness set.

Let us describe the signification of Theorem \ref{thm1.2}. Given
three positive integers $q, p$ and $s$ with $p>\frac{qn}{q-s+1}+n+1$
and $q> s-1\geq 1$, for a $\sigma>0$ such that
$\sigma>\frac{pq-p(s-1)}{(p-2n-1)q-(p-n-1)(s-1)}$, we can choose a
sequence of positive numbers $\{r_m\}$ such that $r_{m+1}>e^{r_m}$
and thus $A(\{r_m\},\sigma)$ is the union of disjoint annuli.
Theorem \ref{thm1.2} tells us that $A(\{r_m\};\sigma)$ is a
$(p,s)$-propagation set for any $q$ linearly non-degenerate
holomorphic curves.
%as long as $q>(s-1)\frac{p-n-1}{p-2n-1}$ and $p>2n+1$.

\begin{cor}\ Let $\{r_m\}$ be an unbounded sequence of
positive numbers. Then two meromorphic functions coincide if they
share five values on $A(\{r_m\};6)$.
\end{cor}

We consider a sequence of ``narrow" annuli.

\begin{thm}\label{thm1.3}\
Let $f_i(z)\ (i=1,2,...,q; q\geq 2)$ be linearly non-degenerate
transcendental holomorphic curves and let $\{r_m\}$ be a sequence of
P\'olya peak of $T(r)$ with the order
$0<\lambda<+\infty.$ If $0<\kappa<1$ and
$\left(1-\kappa^\lambda\right)p>\frac{qn}{q-s+1}+n+1$, then
$\bigcup_{m=1}^\infty A(\kappa r_m,r_m)$ is a $(p,s)$-propagation
set of $f_i(z)\ (i=1,2,...,q)$; Let $\{\kappa_m\}$ be a sequence of
positive numbers
 tending to $0$. If $p-n-1>\frac{qn}{q-s+1}$, then
$\bigcup_{m=1}^\infty A(\kappa_m r_m,r_m)$ is a $(p,s)$-propagation
set of $f_i(z)\ (i=1,2,...,q)$.
\end{thm}

\begin{thm}\label{thm1.9}\
Let $f$ and $g$ be two linearly
non-degenerate transcendental holomorphic curves and let $\{r_m\}$ be a sequence of
P\'olya peak of $T(r)$ with the order
$0<\lambda<+\infty.$ If $0<\kappa<1$ and
$\kappa^\lambda<\frac{3n+2}{5n(2n+3)}$, then
$\bigcup_{m=1}^\infty A(\kappa r_m,r_m)$ is a $(2n+3,2)$-propagation
set of $f$ and $g$; If $\{\kappa_m\}$ be a sequence of
positive numbers
 tending to $0$, then
$\bigcup_{m=1}^\infty A(\kappa_m r_m,r_m)$ is a $(2n+3,2)$-propagation
set of $f$ and $g$.
\end{thm}

The sequence $\{r_m\}$ depends on $T(r)$, in
essence, on $\max\{T(r,f_i):\ i=1,2,...,q\}.$ However, from the
definition of P\'olya peak sequence which will be given in next
section, we know that any subsequence of a P\'olya peak sequence is
still a P\'olya peak sequence. Hence, we can choose a sequence
$\{r_m\}$ of P\'olya peak and a sequence of positive numbers
$\{\kappa_m\}$ in Theorems \ref{thm1.3} and \ref{thm1.9} such that $\kappa_m\to 0$ and
$\frac{\kappa_mr_m}{r_{m-1}}\to\infty (m\to\infty).$ This implies
that $A(\kappa_m r_m,r_m)\cap
A(\kappa_{m-1}r_{m-1},r_{m-1})=\emptyset.$

Finally, we consider the truncated multiplicities
in the uniqueness problem (see \cite{Fujimoto4}, \cite{ThaiQuang},
\cite{Tan}, \cite{TanTruong}, \cite{DethloffTan} and
\cite{YanChen}).

For $n\ge 2$, in \cite{Si2}, Si gave an important development of the technique in studying
uniqueness problem of holomorphic curves and obtained a uniqueness theorem with $2n+2$ hyperplanes,
which shows the unicity of holomorphic curves is different from that of meromorphic functions in essence.
If we consider that $\mathbb{C}$ is replaced by a precise subset $X$, dose this technique work
under the weaker conditions?
By more accurate estimate, we can establish the following result corresponding to
Theorem 2 in Si \cite{Si2}.

\begin{thm}\label{thm1.5}\ Let $f$ and $g$ be two linearly non-degenerate holomorphic
curves of $\mathbb{C}\rightarrow \mathcal{P}^n(\mathbb{C})$ with $n\ge 2$ and at least one of them be of infinite order, and let
$\{\varepsilon_m\}$ be a sequence of positive numbers such that
$\varepsilon_m\rightarrow 0$ as $m\rightarrow \infty$. Then there
exists a sequence of complex numbers $\{z_m\}$ with $z_m\rightarrow
\infty$ as $m\rightarrow \infty$ such that $B(\{z_m\};\{\varepsilon_m\})$ is a $(2n+2,2;n+1)$-propagation set of
$f$ and $g$.
\end{thm}

\begin{thm}\label{thm1.7}\ Let $\{r_m\}$ be a
sequence of positive numbers tending to $\infty$. Then
$\bigcup_{m=1}^\infty A(r_m,e^{r_m})$ is a $(2n+2,2;n+1)$-propagation set of any two linearly
non-degenerate transcendental holomorphic curves $f$ and $g$ from $\mathbb{C}$ to
$\mathcal{P}^n(\mathbb{C})$ with $n\ge 2$.
\end{thm}

We can choose a
sequence of positive numbers $\{r_m\}$ such that $r_{m+1}>e^{r_m}$
and $\bigcup_{m=1}^\infty A(r_m,e^{r_m})$ is also a union of disjoint annuli.

Indeed, by the original result and idea of Si \cite{Si2}, we can weaken the condition
$$\min\{\nu_{\langle{{\bf f}, {\bf a}_j}\rangle}(z),n+1\}=
\min\{\nu_{\langle{{\bf g}, {\bf a}_j}\rangle}(z),n+1\}\ \mbox{on}\ X$$ as $\{z\in f^{-1}(H_j)\cap X:
\nu_{\langle{{\bf f}, {\bf a}_j}\rangle}(z)\leq n\}=\{w\in g^{-1}(H_j)\cap X:\nu_{\langle{{\bf g}, {\bf a}_j}\rangle}(w)\leq
n\}$ and $\{z\in f^{-1}(H_j)\cap X: \nu_{\langle{{\bf f}, {\bf a}_j}\rangle}(z)\ge n\}=\{w\in
g^{-1}(H_j)\cap X:\nu_{\langle{{\bf g}, {\bf a}_j}\rangle}(w)\ge n\}$.

For $n=1$, the result given by Tran \cite{Tan} is an improvement of Nevanlinna's four-value theorem. In
view of Theorem 1.4 in \cite{Tan}, we can obtain the following results
which are corresponding to the four-value theorem.

\begin{thm}\label{thm1.8}\ Let $f$ and $g$ be two distinct meromorphic functions such that $T(r)$ is of infinite order.
Let $a_1,a_2,a_3,a_4$ be four distinct values on $\hat{\mathbb{C}}$. Let $\{\varepsilon_m\}$
be a sequence of positive numbers such that $\varepsilon_m\to 0$ as
$m\to\infty$. Then there exists a sequence of complex numbers
$\{z_m\}$ with $z_m\to\infty$ as $m\to\infty$ such that
if
$$
\min\{\nu_{\langle{\bf f},{\bf a}_j\rangle}(z),2\}=\min\{\nu_{\langle{\bf g},{\bf a}_j\rangle}(z),2\}\ \mbox{on}\ B(\{z_m\};\{\varepsilon_m\})\ \mbox{for}\ j=1,2,3,4,
$$
then $g$ is a M\"obius transformation of $f$.
\end{thm}

\begin{thm}\label{thm1.10}\ Let $\{r_m\}$ be a
sequence of positive numbers tending to $\infty$. For any two transcendental meromorphic functions $f$ and $g$, if
$$
\min\{\nu_{\langle{\bf f},{\bf a}_j\rangle}(z),2\}=\min\{\nu_{\langle{\bf g},{\bf a}_j\rangle}(z),2\}\ \mbox{on}\ A(\{r_m\};\sigma)\ \mbox{for}\ j=1,2,3,4,
$$
then $g$ is a M\"obius transformation of $f$ for $\sigma>866$.\end{thm}

\begin{thm}\label{thm1.11}\ Let $f$ and $g$ be two transcendental meromorphic functions and let $\{r_m\}$ be a sequence of
P\'olya peak of $T(r)$ with the order
$0<\lambda<+\infty.$ If $0<\kappa<1$ and
$\kappa^\lambda<\frac{1}{866}$, then $g$ is a M\"obius transformation of $f$, if
$$
\min\{\nu_{\langle{\bf f},{\bf a}_j\rangle}(z),2\}=\min\{\nu_{\langle{\bf g},{\bf a}_j\rangle}(z),2\}\ \mbox{on}\ \bigcup_{m=1}^\infty A(\kappa r_m,r_m)\ \mbox{for}\ j=1,2,3,4;
$$ Let $\{\kappa_m\}$ be a sequence of
positive numbers
 tending to $0$. Then $g$ is a M\"obius transformation of $f$, if
$$
\min\{\nu_{\langle{\bf f},{\bf a}_j\rangle}(z),2\}=\min\{\nu_{\langle{\bf g},{\bf a}_j\rangle}(z),2\}\ \mbox{on}\ \bigcup_{m=1}^\infty A(\kappa_m r_m,r_m)\ \mbox{for}\ j=1,2,3,4.
$$
\end{thm}

We remark on that corresponding to the most recent study of
uniqueness of holomorphic curves, e.g., in \cite{HaLeSi},
\cite{DethloffSyTan}, \cite{Si}, \cite{SiLe}, \cite{CaoCao} and
\cite{Si1}, one can establish the similar results on the propagation sets.

We will prove Theorems \ref{thm1.1} and \ref{thm1.2} in Section 3, Theorems \ref{thm1.4}
and \ref{thm1.6} in Section 4 and Theorems \ref{thm1.5}, \ref{thm1.8} and \ref{thm1.10} in Section 5.

%Then a subset $X$ of $\mathbb{C}$ is a $(p,s; d)$-propagation set of
%$q$ holomorphic curves $f_i\ (i=1,2,...,q)$, if $X$ is called
%$(p,s)$-propagation set and, additionally,  $
%\min\{\nu_{(f,H_j)},d\}=\min\{\nu_{(g,H_j)},d\}$ on $X$ $(1\leq
%j\leq q)$.

\section{Preliminaries}

Let us begin with the Nevanlinna characteristic of a holomorphic
curve on an angular domain. For $0\leq\alpha<\beta\leq 2\pi$, by
$\Omega(\alpha,\beta)$ we denote the angular domain
$$\Omega(\alpha,\beta):=\{z:\ \alpha<\arg z<\beta\}$$
and $\overline{\Omega}(\alpha,\beta)$ is the closure of
$\Omega(\alpha,\beta)$. Without occurrence of any confusion in the
context, we simply write $\Omega$ for $\Omega(\alpha,\beta)$.
Associated to $\Omega(\alpha,\beta)$ is the quantity
$\omega=\frac{\pi}{\beta-\alpha}.$

Let $f=[f_0:f_1:...:f_n]$ be a holomorphic curve on
$\overline{\Omega}$ and ${\bf{f}}=(f_0,f_1,...,f_n)$ is a reduced
representation of $f$. Set
$$v_f(z):=\bigvee_{j=0}^n\log |f_j(z)|,\ z\in\Omega.$$
Obviously, $v_f(z)$ is subharmonic on $\Omega$. Let $\Delta$ be the
Laplacian. Define the Nevanlinna's characteristic of $f$ on $\Omega$
as
$$S_\Omega(r,f):=\frac{1}{2\pi}\int_1^r\int_\alpha^\beta\left(\frac{1}{t^\omega}-\frac{t^\omega}{r^{2\omega}}\right)
\sin\omega(\theta-\alpha)\Delta v_f(te^{i\theta}).$$ Sometimes, we
also write $S_{\alpha,\beta}(r,f)$ for $S_\Omega(r,f)$.

Set $u_{\bf a}:=\log|\langle{\bf f,a}\rangle|$. Define the counting function of
$f$ with respect to ${\bf a}$ for $\Omega$ as
\begin{eqnarray*}
C_\Omega(r; {\bf
a},f)&=&\frac{1}{2\pi}\int_1^r\int_\alpha^\beta\left(\frac{1}{t^\omega}-\frac{t^\omega}{r^{2\omega}}\right)
\sin\omega(\theta-\alpha)\Delta
u_{\bf a}(te^{i\theta})\\
&=&\sum_{k}\left(\frac{1}{r_k^\omega}-\frac{r_k^\omega}{r^{2\omega}}\right)\sin \omega(\theta_k-\alpha),
\end{eqnarray*}
where $z_k=r_ke^{i\theta_k}$ is a zero of $\langle{\bf f,a}\rangle$ on
$\Omega(r):=B(r)\cap\Omega$, counted with its multiplicities. By
$C^{s)}_{\Omega}(r; {\bf a},f)$ we denote the counting function in
which zero of $\langle{\bf f,a}\rangle$ with multiplicity $p$ is counted by
$\min\{s,p\}$ times. For $r<R$, $C_\Omega(r, R; {\bf a},f)$ is the
counting function for zeros of $\langle{\bf f,a}\rangle$ on $A(r, R)\cap\Omega$.

Set $||{\bf f}||=(|f_0|^2+|f_1|^2+...+|f_n|^2)^{1/2}$ and $||{\bf
a}||=(|a_0|^2+|a_1|^2+...+|a_n|^2)^{1/2}$. The Weil function of $f$
with respect to the hyperplane $H$ with a reduced representation
${\bf a}$ is
$$\lambda_H(f(z)):=\log\frac{||{\bf f}(z)||||{\bf a}||}{|\langle{\bf f}(z),{\bf a}\rangle|}.$$
Define the proximity functions of $f$ for the hyperplane ${\bf a}$
on $\Omega$ by
$$A_\Omega(r; {\bf a},f):=\frac{1}{2\pi}\int_{\Gamma'_r}\lambda_H(f(\zeta))\frac{\partial
w_r}{\partial\bf{n}}{\rm d}s,$$
$$B_\Omega(r; {\bf a},f):=\frac{1}{2\pi}\int_{\Gamma''_r}\lambda_H(f(\zeta))\frac{\partial
w_r}{\partial\bf{n}}{\rm d}s,$$ where $w_r(z)=-{\rm
Im}\left(\frac{1}{(e^{-i\alpha}z)^\omega}+\frac{(e^{-i\alpha}z)^\omega}{r^{2\omega}}\right)$,
$\Gamma''_r=\{re^{i\theta}:\alpha\leq\theta\leq\beta\}\cup\{e^{i\theta}:\alpha\leq\theta\leq\beta\}$
and $\Gamma'_r=\partial\Omega(r)\setminus\Gamma''_r$, and
$$S_\Omega (r; {\bf a},f):=A_\Omega (r; {\bf a},f)+B_\Omega(r; {\bf a},f)+C_\Omega(r; {\bf
a},f).$$

In \cite{Zheng1}, we obtain
\begin{equation}\label{2.1}S_{\alpha,\beta}(r,f)=S_{\alpha,\beta}(r; {\bf a},f)+O(1)\ (r\to\infty)\end{equation}
and establish the following second main theorem for the Nevanlinna
characteristic on an angular domain.

\begin{thm}\label{thm2.1}\ Let $\Omega=\Omega(\alpha,\beta)$ with $0<\beta-\alpha<2\pi$ and $f$ be a
non-degenerate holomorphic curve with a reduced representation ${\bf
f}=(f_0,f_1,...,f_n)$. Let ${\bf a}_1,{\bf a}_2,...,{\bf a}_q$ be $q$ hyperplanes
in general position. Then we have
\begin{eqnarray}\label{2.2}(q-n-1)S_{\alpha,\beta}(r,f)&\leq &
\sum_{k=1}^qC_{\alpha,\beta}(r; {\bf a}_k,f)-C_{\alpha,\beta}(r;
0,W)+R_{\alpha,\beta}(r,f)\nonumber\\
&\leq &\sum_{k=1}^qC^{n)}_{\alpha,\beta}(r; {\bf
a}_k,f)+R_{\alpha,\beta}(r,f),\end{eqnarray} where $W$ is the
Wronskian of $f_0,f_1,...,f_n$ and $R_{\alpha,\beta}(r,f)$ is called
the error term with the estimate
\begin{equation}\label{2.3}R_{\alpha,\beta}(r,f)\leq
K\omega(\log^+T(r,f)+\log^+r+1),\end{equation} for all $r>1$ but
possibly a set of $r$ with finite linear measure, where $K$ is a
constant independent of $\omega$.
\end{thm}

% If $g$ is meromorphic on the
%complex plane, then a nice estimate in terms of $\log T(r,g)$ was
%given by Goldberg and Ostrovskii in \cite{GoldbergeOstrovskii}:
%$$A_{\alpha,\beta}\left(r,\frac{g'}{g}\right)\leq K\left[\left(\frac{R}{r}\right)^\omega\int_1^R\frac{\log T(t,g)}{
%t^{1+\omega}} dt+\log\frac{r}{R-r}+\log\frac{R}{r}\right]$$ and
%$$B_{\alpha,\beta}\left(r,\frac{g'}{g}\right)\leq \frac{4\omega}{r^\omega}m\left(r,\frac{g'}{g}\right).$$
%An important result is that if $\int^\infty
%t^{-1-\omega}\log^+T(t,g)\mathrm{d}t<\infty$ with
%$\omega=\frac{\pi}{\beta-\alpha}$, then
%$(A+B)_{\alpha,\beta}\left(r,\frac{g'}{g}\right)=O(1)\
%(r\rightarrow\infty).$

%\begin{lem}\label{lem2.4} \ Let $h_0,h_1,...,h_n$ be $n+1$ analytic
%functions on
%$\overline{\Omega}(\alpha-\varepsilon,\beta+\varepsilon)$. Set
%$$\widehat{W}(z)=\frac{W(h_0,h_1,...,h_n)}{h_0h_1\cdots h_n}$$ and
%$h=[h_0:h_1:...:h_n].$ Then
%$$(A+B)_{\alpha,\beta}(r,\widehat{W})\leq K(\log^+S_{\alpha-\varepsilon,\beta+\varepsilon}(r,h)+\log^+r+1),$$
%for all $r>1$ but possibly a set of $r$ with finite linear measure,
%where $K$ is a constant depending on $\varepsilon$.
%\end{lem}

Let $\mu_f$ be the Riesz measure of $v_f(z)$. In fact,
$\mu_f=\frac{1}{2\pi}\Delta v_f$. Define
$$\mathcal{A}(r,\Omega, f)=\mu_f(\overline{B(r)\cap\Omega})$$
and
$$\mathcal{T}(r,\Omega,f)=\int_0^r\frac{\mathcal{A}(t,\Omega,f)}{t}\mathrm{d}t.$$
$\mathcal{A}(r,\Omega,f)$ is called the unintegrated Ahlfors-Shimizu
characteristic of holomorphic curve $f$ on $\Omega$ and
$\mathcal{T}(r,\Omega,f)$ is called the Ahlfors-Shimizu
characteristic of holomorphic curve $f$ on $\Omega$. Then
$T(r,f)=\mathcal{T}(r,\mathbb{C},f).$ In the sequel, we simply write
$\mathcal{A}(r,f)$ and $\mathcal{T}(r,f)$ for
$\mathcal{A}(r,\mathbb{C},f)$ and $\mathcal{T}(r,\mathbb{C},f)$.

$u_{\bf{a}}$ is subharmonic on $\Omega$ for a hyperplane ${\bf a}$.
By $\mu^{\bf a}$ we denote the Riesz measure of $u_{\bf{a}}$. Set
$$n_\Omega(r,{\bf{a}},f)=\mu^{\bf{a}}(\overline{B(r)\cap\Omega}),$$
and hence $n_\Omega(r,{\bf{a}},f)$ is the number of zeros of $\langle{\bf
f}(z),{\bf a}\rangle$ in $B(r)\cap\Omega.$ Define
$$N_\Omega(r,{\bf{a}},f)=\int_1^r\frac{n_\Omega(t,{\bf{a}},f)}{t}\mathrm{d}t.$$

In \cite{Zheng1}, we establish the second main theorem for
Ahlfors-Shimizu characteristic and applying this main theorem we
confirm the existence of Borel directions and $T$-directions. In
view of the same argument as in \cite{Zheng}, we can compare
$N_\Omega(r,{\bf{a}},f)$ to $C_\Omega(r;{\bf{a}},f)$ and
$\mathcal{T}(r,\Omega,f)$ to $S_{\Omega}(r,f)$.

\begin{lem}\label{lem2.1}\ Let $f$ be a holomorphic curve and ${\bf a}$ a hyperplane. Then
$$C_{\Omega}(r;{\bf a},f)\geq \omega\sin(\omega\varepsilon)\frac{N_{\Omega_\varepsilon}(r,{\bf a},f)}{r^\omega}$$
for $\varepsilon>0$, where $\Omega_{\varepsilon}=\Omega(\alpha+\varepsilon,\beta-\varepsilon)$; and
$$C_\Omega(r; {\bf{a}},f)\leq 2\omega\frac{N_\Omega(r,{\bf{a}},f)}{r^\omega}+\omega^2\int_1^r
\frac{N_\Omega(t,{\bf{a}},f)}{t^{\omega+1}}{\rm d}t.$$
\end{lem}

For $q$ holomorphic curves $f_i$ with a reduced representation ${\bf
f}_i\ (i=1,2,...,q)$, if ${\bf f}_1\wedge{\bf
f}_2\wedge\cdots\wedge{\bf f}_q\not\equiv 0$, then the divisor
$\nu_{f_1\wedge f_2\wedge\cdots\wedge f_q}\geq 0$ associated with
${\bf f}_1\wedge{\bf f}_2\wedge\cdots\wedge{\bf f}_q$ exists.
Obviously $\nu_{f_1\wedge f_2\wedge\cdots\wedge f_q}$ is independent
of the choice of the reduced representation ${\bf f}_i$ of $f_i$. We
write $C_\Omega(r;\nu_{f_1\wedge f_2\wedge\cdots\wedge f_q})$ and
$N_\Omega(r,\nu_{f_1\wedge f_2\wedge\cdots\wedge f_q})$ for the
corresponding counting functions.

The following is the first main theorem for the wedge product.

\begin{lem}\label{lem2.2}\ Then
\begin{equation}\label{2.4}N(r,\nu_{f_1\wedge f_2\wedge\cdots\wedge
f_q})\leq\sum_{j=1}^q T(r, f_j)+O(1).
\end{equation}
For an angular domain $\Omega$, we have
\begin{equation}\label{2.5}C_\Omega(r;\nu_{f_1\wedge f_2\wedge\cdots\wedge f_q})\leq
\sum_{j=1}^qS_\Omega(r, f_j)+O(1).\end{equation}\end{lem}

(\ref{2.4}) can be found in Stoll \cite{Stoll} and (\ref{2.5}) can
be obtained by the same method as there. In fact, in view of the
definition of the exterior product, Lemma \ref{lem2.2} follows from
the direct calculation.

\begin{lem}\label{lem2.3}
 Assume that (\ref{1.1}) holds on $X=A(r,R)\cap\Omega$. If for $1\leq t\leq q$,
 $X\cap f_t^{-1}(H_i)\cap f_t^{-1}(H_j)=\emptyset (i\not=j)$, then we have
$$\sum_{j=1}^pC^{n)}_\Omega(r,R;H_j,f_t)\leq \frac{n}{q-s+1}C_\Omega(r,R;\nu_{f_1\wedge f_2\wedge\cdots\wedge f_q}).$$
The inequality also holds with $N$ replacing $C$.
\end{lem}

A direct calculation yields Lemma \ref{lem2.3}, please see Page 112
of \cite{Ru} where it is proved that a zero of $\langle{\bf f}_t,{\bf
a}_j\rangle$ is a zero of ${\bf f}_1\wedge {\bf f}_2\wedge\cdots\wedge
{\bf f}_q$ with multiplicity at least $q-s+1$.

Let $T$ be a non-negative and non-decreasing continuous function in
$(0,+\infty)$. A positive increasing unbounded sequence $\{r_m\}$ is
a sequence of P\'olya peak of order $\sigma$ of $T(r)$, if there
exist sequences $\{r_m'\}, \{r_m''\}$, $\{\varepsilon_m\}$ and
$\{\varepsilon'_m\}$ such that

1) $r'_m\rightarrow\infty, r_m/r'_m\rightarrow\infty,
r''_m/r_m\rightarrow\infty, \varepsilon_m\rightarrow 0$ and
$\varepsilon'_m\rightarrow 0$, as $m\rightarrow\infty;$

2) $T(t)\leq (1+\varepsilon_m)\left(\frac{t}{r_m}\right)^\sigma
T(r_m),\ r'_m\leq t\leq r''_m;$

3) $T(t)/t^{\sigma-\varepsilon'_m}\leq
KT(r_m)/r_m^{\sigma-\varepsilon'_m},\ 1\leq t\leq r''_m,$ for a
positive constant $K$.

If $T$ has the lower order $\tau<\infty$ and order
$0<\lambda\leq\infty$, then for a finite positive number $\sigma$
with $\tau\leq\sigma\leq\lambda$ and a set $E$ of positive numbers
with finite logarithmic measure there must be a sequence of P\'olya
peak $\{r_m\}$ with order $\sigma$ of $T$ outside $E$ (Theorem 1.1.3
of \cite{Zheng}). The P\'olya peak sequence was first introduced by
Edrei, please see references in \cite{Zheng}. For a positive
increasing real-valued function $T$ of infinite order, the following
result is established in \cite{Zheng}.

\begin{lem}\label{lem2.4}\ Let $T$ be a positive increasing continuous
real-valued function $T$ of infinite order and $F$ a set of positive
real numbers with finite logarithmic measure. Then given a sequence
$\{s_n\}$ of positive numbers, there exists a sequence $\{r_n\}$ of
positive real numbers outside $F$ tending to $\infty$ such that
$$\frac{T(t)}{t^{s_n}}\leq e\frac{T(r_n)}{r_n^{s_n}},\ 1\leq t\leq
r_n.$$
\end{lem}

Let $f$ be a convex function on $[0,\infty)$. If $f(0)=0$, then
$\frac{f(x)}{x}$ is increasing on $(0,\infty)$. Therefore, if $T(r)$
is logarithmic convex on $[1,\infty)$, then for $1<r'<r$ we have
$$T(r)\geq \frac{\log r}{\log r'}T(r')+\frac{\log r'-\log r}{\log
r'}T(1).$$ In particular, for $\sigma>1$ and $r>1$ we have
\begin{equation}\label{2.6}T(r^\sigma)\geq \sigma
T(r)+(1-\sigma)T(1).\end{equation}

\section{Proofs of Theorems \ref{thm1.1} and \ref{thm1.2}}

{\sl Proof of Theorem \ref{thm1.1}.}\ Suppose that (\ref{11}) does
not hold. Take a $\theta\in [0,2\pi)$ and set
$\mathcal{Z}_\delta(\theta)=\{z:\ \theta-\delta<\arg
z<\theta+\delta\}$ for $\delta>0$. If no confusion occurs, we simply
write $\mathcal{Z}_\delta$. It is easy to see that for any
$z=re^{i\theta}$ and any $0<\varepsilon<\frac{3}{2\pi}$, we can find
that for $\delta=\frac{1}{2}\varepsilon$ and
$\kappa^2=\sqrt{1-\frac{2\pi}{3}\varepsilon}$,
$$\mathcal{Z}_\delta\cap A(\kappa^2 r,r)\subset B(z,\varepsilon|z|).$$
Suppose that $A_j=B(z,\varepsilon|z|)\cap f^{-1}_i(H_j)
(i=1,2,...,q)$ and (\ref{1.1}) holds on $\bigcup_{j=1}^pA_j.$
Applying Theorem \ref{thm2.1} to the angular domain
$\mathcal{Z}_\delta$, setting $u^i_{H_j}=\log|\langle {\bf
f}_i,H_j\rangle|$, we have
\begin{eqnarray*}&\ & (p-n-1)S_{\mathcal{Z}_\delta}(r,f_i)\\
&\leq &\sum_{j=1}^pC^{n)}_{\mathcal{Z}_\delta}(r; H_j,f_i)+R_{\mathcal{Z}_\delta}(r,f_i)\\
&\leq &\sum_{j=1}^pC^{n)}_{\mathcal{Z}_\delta}(\kappa^2 r,r; H_j,f_i)
+\sum_{j=1}^pC^{n)}_{\mathcal{Z}_\delta}(\kappa^2 r;H_j,f_i)\\
&&\!\!\!+\sum_{j=1}^p\frac{1}{2\pi}\int_1^{\kappa^2r}\int_{\theta-\delta}^{\theta+\delta}
\left(\frac{t^\omega}{(\kappa^2r)^{2\omega}}-\frac{t^\omega}{r^{2\omega}}\right)
\sin\omega(\vartheta-\theta+\delta)\Delta
u^i_{H_j}(te^{i\vartheta})+R_{\mathcal{Z}_\delta}(r,f_i)\end{eqnarray*}
\begin{eqnarray*}
&\leq&\frac{n}{q-s+1}C_{\mathcal{Z}_\delta}(\kappa^2 r,r;
\nu_{f_1\wedge f_2\wedge\cdots\wedge
f_q})\ \ (\text{by\ Lemma\ \ref{lem2.3}})\\
&\ &+\sum_{j=1}^pC^{n)}_{\mathcal{Z}_\delta}(\kappa r;H_j,f_i)
+\frac{\kappa^{-4\omega}-1}{2\pi
r^{2\omega}}\sum_{j=1}^p\int_1^{\kappa^2r}\int_{\theta-\delta}^{\theta+\delta}
t^\omega\sin\omega(\vartheta-\theta+\delta)\Delta
u^i_{H_j}(te^{i\vartheta})\\
&\ &+R_{\mathcal{Z}_\delta}(r,f_i)\\
&\leq&\frac{n}{q-s+1}C_{\mathcal{Z}_\delta}(r; \nu_{f_1\wedge
f_2\wedge\cdots\wedge
f_q})+\sum_{j=1}^pC^{n)}_{\mathcal{Z}_\delta}(\kappa r;H_j,f_i)\\
&\ &+\frac{\kappa^{2\omega}(\kappa^{-4\omega}-1)}{
r^{\omega}}\sum_{j=1}^pn_{\mathcal{Z}_\delta}(\kappa^2 r,H_j,f_i)
+R_{\mathcal{Z}_\delta}(r,f_i)\\
 &< &\frac{n}{q-s+1}\sum_{i=1}^q
S_{\mathcal{Z}_\delta}( r,f_i)\ \
(\text{by\ Lemma\ \ref{lem2.2}})\\
 &\
&+\sum_{j=1}^pC^{n)}_{\mathcal{Z}_\delta}(\kappa
r;H_j,f_i)+\frac{\kappa^{-4\omega}-1}{r^{\omega}\log
\kappa^{-1}}\sum_{j=1}^pN_{\mathcal{Z}_\delta}(\kappa r,H_j,f_i)
+R_{\mathcal{Z}_\delta}(r,f_i).
\end{eqnarray*}
This implies that
\begin{eqnarray*}
(p-n-1)\sum_{i=1}^q S_{\mathcal{Z}_\delta}(r,f_i)&\leq &\frac{q
n}{q-s+1}\sum_{i=1}^q
S_{\mathcal{Z}_\delta}(r,f_i)+\sum_{i=1}^q\sum_{j=1}^pC^{n)}_{\mathcal{Z}_\delta}(\kappa r;H_j,f_i)\\
&\ & +\frac{\kappa^{-4\omega}-1}{r^{\omega}\log
\kappa^{-1}}\sum_{i=1}^q\sum_{j=1}^pN_{\mathcal{Z}_\delta}(\kappa
r,H_j,f_i)+\sum_{i=1}^q R_{\mathcal{Z}_\delta}(r,f_i),
\end{eqnarray*}
so that $$Q\sum_{i=1}^q S_{\mathcal{Z}_\delta}(r,f_i)\leq
\frac{1}{p-n-1}\sum_{i=1}^q\sum_{j=1}^pC^{n)}_{\mathcal{Z}_\delta}(\kappa
r;H_j,f_i)$$
\begin{equation}\label{1.2}+\frac{1}{p-n-1}\left(\frac{\kappa^{-4\omega}-1}{
r^{\omega}\log
\kappa^{-1}}\sum_{i=1}^q\sum_{j=1}^pN_{\mathcal{Z}_\delta}(\kappa
r,H_j,f_i)+\sum_{i=1}^q
R_{\mathcal{Z}_\delta}(r,f_i)\right),\end{equation} where
$Q=1-\frac{qn}{(p-n-1)(q-s+1)}>0$. In view of Lemma \ref{lem2.1}, we
have
$$C^{n)}_{\mathcal{Z}_\delta}(\kappa r;H_j,f_i)\leq
2\omega\frac{N(\kappa r,H_j,f_i)}{(\kappa
r)^{\omega}}+\omega^2\int_1^{\kappa
r}\frac{N(t,H_j,f_i)}{t^{\omega+1}}{\rm d}t$$
$$\leq \frac{2\omega}{\kappa^\omega}
\frac{T(\kappa r,f_i)}{r^{\omega}}+\omega^2\int_1^{\kappa
r}\frac{T(t,f_i)}{t^{\omega+1}}{\rm d}t+O(1), \
\omega=\frac{\pi}{2\delta}.$$

 If for some $s>\omega$ we have
\begin{equation}\label{1.3} \frac{T(t)}{t^s}\leq e\frac{T(r)}{r^s},\ 1\leq
t\leq r,
\end{equation}
then
$$\int_1^{r}\frac{T(t)}{t^{\omega+1}}{\rm
d}t\leq e\int_1^{r}\frac{T(r)}{r^s}\frac{t^s}{t^{\omega+1}}{\rm
d}t=e\frac{T(r)}{r^s}\frac{1}{s-\omega}t^{s-\omega}|_1^r$$
$$<\frac{e}{s-\omega}\frac{T(r)}{r^\omega}.$$
Thus if (\ref{1.3}) holds for $T(r)=\sum_{i=1}^q T(r,f_i)$, we have
\begin{eqnarray}\label{1.4}
\sum_{i=1}^qC^{n)}_{\mathcal{Z}_\delta}(\kappa r;H_j,f_i)&\leq &
\frac{2\omega}{\kappa^\omega} \frac{T(\kappa r)}{r^{\omega}}+
\frac{e\omega^2}{s-\omega}\frac{T(r)}{r^\omega}+O(1)\nonumber\\
&\leq& \left(2e\omega\kappa^{s-\omega}+
\frac{e\omega^2}{s-\omega}\right)\frac{T(r)}{r^\omega}+O(1).\end{eqnarray}
And we have
$$\frac{\kappa^{-4\omega}-1}{
r^{\omega}\log \kappa^{-1}}\sum_{i=1}^qN_{\mathcal{Z}_\delta}(\kappa
r,H_j,f_i)<\frac{e(\kappa^{-4\omega}-1)\kappa^s}{\log
\kappa^{-1}}\frac{T(r)}{r^{\omega}}.$$

 On the other hand, we have, for any hyperplane $H$,
\begin{eqnarray}\label{3.4}S_{\mathcal{Z}_\delta}(r,f_i)&\geq &
C_{\mathcal{Z}_\delta}(r;H,f_i)+O(1)\geq
\omega\sin(\omega\frac{\delta}{2})\frac{N_{\mathcal{Z}_{\delta/2}}(r,H,f_i)}{r^\omega}+O(1)\nonumber\\
&=&\frac{\sqrt{2}\omega}{2}\frac{N_{\mathcal{Z}_{\delta/2}}(r,H,f_i)}{r^\omega}+O(1).\end{eqnarray}
It follows from (\ref{1.2}), (\ref{1.4}) and (\ref{3.4}) that
\begin{eqnarray}\label{1.5}
&\ &\frac{\sqrt{2}\omega}{2} Q\sum_{i=1}^qN_{\mathcal{Z}_{\delta/2}}(r,H,f_i)\nonumber\\
&\leq &\frac{p}{p-n-1}\left(2e\omega\kappa^{s-\omega}+
\frac{e\omega^2}{s-\omega}+\frac{e(\kappa^{-4\omega}-1)\kappa^s}{
\log
\kappa^{-1}}\right)T(r)\nonumber\\
&\
&+\frac{r^{\omega}}{p-n-1}\sum_{i=1}^qR_{\mathcal{Z}_\delta}(r,f_i).\end{eqnarray}

Take a sequence of positive numbers $\{s_m\}$ such that
$s_m\to\infty$ and for the given sequence of positive numbers
$\{\varepsilon_m\}$, $s_m\varepsilon_m^2\to\infty$ as $m\to\infty$.
Then set $\delta_m=\frac{1}{2}\varepsilon_m$,
$\kappa^2_m=\sqrt{1-\frac{2\pi}{3}\varepsilon_m}$,
$\omega_m=\frac{\pi}{2\delta_m}$ and
$\mathcal{Z}_{\delta_m}^k=\mathcal{Z}_{\delta_m}(\theta_{m,k}),\
\theta_{m,k}=\frac{(2k-1)}{4}\varepsilon_m\
(k=1,2,...,\left[\frac{4\pi}{\varepsilon_m}\right]+2)$.

Since $T(r)=\sum_{i=1}^q T(r,f_i)$ is of infinite order, in view of
Lemma \ref{lem2.4}, there exists an increasing sequence of positive
numbers $\{r_m\}$ such that (\ref{1.3}) holds for $T(r)$ with every
$r=r_m$ and every $s=s_m$. From Cartan's second main theorem (see
Theorem A3.1.7 of \cite{Ru}), we can have a hyperplane $H$ such that
$$\sum_{i=1}^qN(r_m,H,f_i)\sim T(r_m)\ (m\to\infty).$$
Since
$$\sum_{i=1}^qN(r_m,H,f_i)\leq\sum_{k=1}^{\left[\frac{4\pi}{\varepsilon_m}\right]+2}
\sum_{i=1}^qN_{\mathcal{Z}^k_{\delta_m/2}}(r_m,H,f_i)$$
$$\leq \left(\left[\frac{4\pi}{\varepsilon_m}\right]+2\right)
\sum_{i=1}^qN_{\mathcal{Z}^{k_m}_{\delta_m/2}}(r_m,H,f_i) , \forall\
m\in\mathbb{N}, $$ for some $k_m$, we have
$$\frac{\varepsilon_m}{4\pi+1}T(r_m)\leq \sum_{i=1}^qN_{\mathcal{Z}^{k_m}_{\delta_m/2}}(r_m,H,f_i).$$
From (\ref{1.5}) it follows that
\begin{eqnarray}\label{3.5}\frac{\sqrt{2}}{10} Q T(r_m)&\leq&
\frac{\sqrt{2}}{2}\omega_m Q \frac{\varepsilon_m}{4\pi+1}T(r_m)\nonumber\\
&\leq &\frac{p}{p-n-1}\left(2e\omega_m\kappa_m^{s_m-\omega_m}+
\frac{e\omega_m^2}{s_m-\omega_m}+\frac{e(\kappa_m^{-4\omega_m}-1)\kappa^{s_m}_m}{
\log \kappa_m^{-1}}\right)T(r_m)\nonumber\\
&\
&+\frac{r_m^{\omega_m}}{p-n-1}\sum_{i=1}^qR_{\mathcal{Z}_{\delta_m}^{k_m}}(r_m,f_i),
\ s_m>\omega_m.
\end{eqnarray}
In view of (\ref{1.3}), $T(r_m)\geq \frac{1}{e}r_m^{s_m}T(1)$, which reduces
$$T(r_m)= T^{1/2}(r_m)T^{1/2}(r_m)\geq
\frac{1}{e^{1/2}}T^{1/2}(r_m)r_m^{s_m/2}T^{1/2}(1).$$ Therefore in view of
(\ref{2.3}) we have
$$\frac{r_m^{\omega_m}}{p-n-1}\sum_{i=1}^qR_{\mathcal{Z}_{\delta_m}^{k_m}}(r_m,f_i)\leq
K\omega_m r_m^{\omega_m-s_m/2}T(r_m),\ \forall\ m\in\mathbb{N},$$
where $K$ is a constant independent of $m$. Substituting the above
inequality to (\ref{3.5}) yields
$$\frac{\sqrt{2}}{10}Q\leq\frac{p}{p-n-1}\left(2e\omega_m\kappa_m^{s_m-\omega_m}+
\frac{e\omega^2_m}{s_m-\omega_m}+\frac{e(\kappa_m^{-4\omega_m}-1)\kappa^{s_m}_m}{
\log \kappa_m^{-1}}\right)$$
\begin{equation}\label{3.6}+K\omega_m
r_m^{\omega_m-s_m/2},\end{equation} $\forall\ s_m>\omega_m.$  It is
easily seen that as $m\to\infty$, we have
\begin{eqnarray*}
&\
&\frac{e\omega^2_m}{s_m-\omega_m}=\frac{e\pi^2}{s_m\varepsilon^2_m-\pi\varepsilon_m}\to
0,\\
&\ & (s_m-\omega_m)\log
\kappa_m-\log\varepsilon_m=\frac{s_m\varepsilon_m-\pi}{\varepsilon_m}\frac{1}{4}\left(-\frac{2\pi}{3}\varepsilon_m
+o(\varepsilon_m)\right)-\log\varepsilon_m\\
&\ &=\left(-\frac{\pi}{6}+o(1)\right)(s_m\varepsilon_m-\pi)-\log\varepsilon_m\\
&\
&=\left(-\frac{\pi}{6}+o(1)\right)(s_m\varepsilon_m-\pi)-s_m\varepsilon_m\frac{\varepsilon_m\log\varepsilon_m}{
s_m\varepsilon_m^2}\end{eqnarray*}
\begin{eqnarray*}
&\ &=\left(-\frac{\pi}{6}+o(1)\right)s_m\varepsilon_m\to -\infty,\\
&\ & 2e\omega_m\kappa_m^{s_m-\omega_m}\to 0,\ \omega_m
r_m^{\omega_m-s_m/2}\to 0,\\
&\
&\kappa_m^{-4\omega_m}=\exp\left(-\frac{\pi}{\varepsilon_m}\log\left(1-\frac{2\pi}{3}\varepsilon_m\right)\right)
\to\exp\frac{2\pi^2}{3},\\
&\ &
\frac{\kappa_m^{s_m}}{-\log\kappa_m}=\frac{4\kappa_m^{s_m}}{-\log\left(1-\frac{2\pi}{3}\varepsilon_m\right)}
\sim\frac{6}{\pi}\frac{\kappa_m^{s_m}}{\varepsilon_m}\to
0.\end{eqnarray*}

Therefore, if the conditions in Theorem \ref{thm1.1} holds for
$B(\{z_m\};\{\varepsilon_m\})$ with
$z_m=r_me^{i\theta_{m,k_m}}$, then we have (\ref{3.6}).
 Letting
$m\to\infty$, the right side of (\ref{3.6}) will tend to $0$ and so
$Q=0$, a contradiction is derived. \qed

\

{\sl Proof of Theorem \ref{thm1.2}.}\ Suppose that (\ref{11}) does
not hold. Take $r'_m$ in $[r_m/2,r_m]$ outside the exceptional set
in Cartan's second main theorem which only has finite linear measure
for $q$ given holomorphic curves. Below we always take $r=r'_m$.
Then
\begin{eqnarray*}&\ & (p-n-1)T(r^\sigma,f_i)\\
&\leq &\sum_{j=1}^pN^{n)}(r^\sigma,H_j,f_i)+O(\log^+rT(r^\sigma,f_i))\\
&\leq&\frac{n}{q-s+1}N(r^\sigma,\nu_{f_1\wedge f_2\wedge\cdots\wedge
f_q})\ \
(\text{by\ Lemma\ \ref{lem2.3}})\\
&\
&+\sum_{j=1}^pN^{n)}(2r,H_j,f_i)+O(\log^+rT(r^\sigma,f_i))\\
&\leq &\frac{n}{q-s+1}\sum_{i=1}^q T(r^\sigma,f_i)\ \ (\text{by\
Lemma\ \ref{lem2.2}})\\
&\ &+pT(2r,f_i)+O(\log^+rT(r^\sigma,f_i)).
\end{eqnarray*}
This implies that
\begin{eqnarray*}
(p-n-1)\sum_{i=1}^q T(r^\sigma,f_i)&\leq &\frac{q
n}{q-s+1}\sum_{i=1}^q T(r^\sigma,f_i)+p\sum_{i=1}^q
T(2r,f_i)\\
&\ &+O(\sum_{i=1}^q\log^+rT(r^\sigma,f_i)),
\end{eqnarray*}
so that
$$\left(1-\frac{qn}{(p-n-1)(q-s+1)}+o(1)\right)\sum_{i=1}^q
T(r^\sigma,f_i)\leq \frac{p}{p-n-1}\sum_{i=1}^q T(2r,f_i),$$
$m\to\infty.$ In view of the logarithmic convex of $\sum_{i=1}^q
T(r,f_i)$, we have
$$\sum_{i=1}^q
T(r^\sigma,f_i)\geq \frac{\sigma\log r}{\log 2r}\sum_{i=1}^q
T(2r,f_i)+(1-\sigma)\sum_{i=1}^q T(1,f_i).$$ It is easy to see that
$$\sigma\left(1-\frac{qn}{(p-n-1)(q-s+1)}\right)\leq\frac{p}{p-n-1},$$
which reduces
$$p-\frac{p}{\sigma}\leq \frac{qn}{q-s+1}+n+1.$$
A contradiction is derived. \qed

The proof of Theorem \ref{thm1.3} is similar to that of Theorem
\ref{thm1.2}. We would like to leave it to the reader.

\section{Proofs of Theorems \ref{thm1.4} and \ref{thm1.6}}

Let $f_1,f_2:{\mathbb{C}}\rightarrow{\mathcal{P}}^n({\mathbb{C}})$
be two holomorphic curves with reduced representations ${\bf f}_1,{\bf f}_2$.  Let $H_1,H_2,...,H_p$ (or ${\bf a}_1,{\bf a}_2,...,{\bf a}_p$) be $p (>2n)$
hyperplanes in ${\mathcal{P}}^n({\mathbb{C}})$ located in general
position. Suppose that $f_1\not\equiv f_2$. By changing indices if
necessary, we may assume that
\begin{eqnarray*}
&&\!\!\!\!\!\!\!\!\underbrace{\frac{\langle{\bf f}_1,{\bf
a}_1\rangle}{\langle{\bf f}_2,{\bf
a}_1\rangle}\equiv\frac{\langle{\bf f}_1,{\bf
a}_2\rangle}{\langle{\bf f}_2,{\bf a}_2\rangle}\equiv
...\equiv\frac{\langle{\bf f}_1,{\bf
a}_{\varsigma_1}\rangle}{\langle{\bf f}_2,{\bf
a}_{\varsigma_1}\rangle}}_{\mbox{group}\ 1}\not\equiv
\underbrace{\frac{\langle{\bf f}_1,{\bf
a}_{\varsigma_1+1}\rangle}{\langle{\bf f}_2,{\bf
a}_{\varsigma_1+1}\rangle}\equiv
...\equiv\frac{\langle{\bf f}_1,{\bf a}_{\varsigma_2}\rangle}{\langle{\bf f}_2,{\bf a}_{\varsigma_2}\rangle}}_{\mbox{group}\  2}\\
&&\not\equiv...\not\equiv\underbrace{\frac{\langle{\bf f}_1,{\bf
a}_{\varsigma_{u-1}+1}\rangle}{\langle{\bf f}_2,{\bf
a}_{\varsigma_{u-1}+1}\rangle}\equiv ...\equiv\frac{\langle{\bf
f}_1,{\bf a}_{\varsigma_u}\rangle}{\langle{\bf f}_2,{\bf
a}_{\varsigma_u}\rangle}}_{\mbox{group}\ u},
\end{eqnarray*}
where $\varsigma_u=p$. Since $f_1\not\equiv f_2$, there exist at
most $n$ elements in every group.

We define the map $\sigma:\{1,2,...,p\}\rightarrow\{1,2,...,p\}$
by
$$
\sigma(i)=\left\{\begin{array}{cl}i+n, &\mbox{if}\ i\le p-n;\\
i+n-p, &\mbox{if}\ i>p-n.
\end{array}\right.
$$
It is easy to see that $\sigma$ is bijective and $|\sigma(i)-i|\ge
n$ (note that $p>2n$). Hence $\frac{\langle{\bf f}_1,{\bf
a}_i\rangle}{\langle{\bf f}_2,{\bf a}_i\rangle}$ and
$\frac{\langle{\bf f}_1,{\bf a}_{\sigma(i)}\rangle}{\langle{\bf
f}_2,{\bf a}_{\sigma(i)}\rangle}$ belong to distinct groups, so that
$$
P_i=\langle{\bf f}_1,{\bf a}_i\rangle\langle{\bf f}_2,{\bf
a}_{\sigma(i)}\rangle-\langle{\bf f}_2,{\bf a}_{i}\rangle\langle{\bf
f}_1,{\bf a}_{\sigma(i)}\rangle\not\equiv 0\ {\mbox{for}}\
i=1,2,...,p.
$$

\

{\sl Proof of Theorem \ref{thm1.4}.}\ Suppose $f\not\equiv g$ and
set $f_1:=f$ and $f_2:=g$. So for $i=1,2,...,2n+3$, $P_i\not\equiv
0$. For $z=re^{i\theta}$, assume that $f_1$ and $f_2$ satisfy

i) $ \min\{\nu_{\langle{\bf f}_1,{\bf a}_j\rangle},1\}=\min\{\nu_{\langle{\bf f}_2,{\bf a}_j\rangle},1\}$ on
$B(z,\varepsilon|z|)$ for $j=1,2,..., 2n+3$;

ii) $ f_1=f_2$ on $B(z,\varepsilon|z|)\cap\bigcup_{j=1}^{2n+3}
f^{-1}(H_j)$.

\noindent Below we use the notations $\delta, \kappa$ and
$\mathcal{Z}_{\delta}=\mathcal{Z}_{\delta}(\theta)$ in the proof of
Theorem \ref{thm1.1} with the same meanings.

\begin{lem}\label{lem4.2}\ We have
\begin{eqnarray}&&\left(1+\frac{n+2}{2n}\right)\sum_{l=1}^2S_{\mathcal{Z}_{\delta}}(r,f_{l})\le
\left(2+\frac{1}{2n}\right)\sum_{l=1}^2R_{\mathcal{Z}_{\delta}}(r,f_l)\nonumber\\
&\ &+
\left(2+\frac{1}{2n}\right)\sum_{j=1}^{2n+3}\sum_{l=1}^2C_{\mathcal{Z}_{\delta}}^{n)}(\kappa
r;H_j,f_{l})+O(1)\label{eqn00-1-4}\\
&\ &+\left(2+\frac{1}{2n}\right)\frac{\kappa^{-4\omega}-1}{
r^{\omega}\log \kappa^{-1}}\sum_{j=1}^{2n+3}\sum_{l=1}^2N_{\mathcal{Z}_\delta}(\kappa
r,H_j,f_l).\nonumber
\end{eqnarray}
\end{lem}

\begin{proof}By the assumptions, we have, for $i=1,2,...,2n+3$,
\begin{eqnarray*}
&\ &\nu_{P_i}(z_0)\ge\min\{\nu_{\langle{\bf f}_1,{\bf a}_i\rangle}(z_0),\nu_{\langle{\bf f}_2,{\bf a}_i\rangle}(z_0)\}\\
&+&\min\{\nu_{\langle{\bf f}_1,{\bf a}_{\sigma(i)}\rangle}(z_0),\nu_{\langle{\bf f}_2,{\bf a}_{\sigma(i)}\rangle}(z_0)\}
+\sum_{\substack{j=1\\j\neq
i,\sigma(i)}}^{2n+3}\nu^{1)}_{\langle{\bf f}_l,{\bf a}_j\rangle}(z_0)
\end{eqnarray*}
for $z_0\in B(z,\varepsilon|z|)$ and $l=1,2$. Note that
$$
\min\{a,b\}\ge\min\{a,n\}+\min\{b,n\}-n
$$
for any positive integers $a$ and $b$. Hence,
\begin{eqnarray*}
\nu_{P_i}(z_0)&\ge&\sum_{j= i,\sigma(i)}(\min\{\nu_{\langle{\bf f}_1,{\bf a}_j\rangle}(z_0),n\}
+\min\{\nu_{\langle{\bf f}_2,{\bf a}_j\rangle}(z_0),n\}\\
&\ &-n\min\{\nu_{\langle{\bf f}_l,{\bf a}_j\rangle}(z_0),1\})+\sum_{\substack{j=1\\j\neq
i,\sigma(i)}}^{2n+3}\nu^{1)}_{\langle{\bf f}_l,{\bf a}_j\rangle}(z_0).
\end{eqnarray*}
It follows that $$\sum_{\substack{j=1\\j\neq
i,\sigma(i)}}^{2n+3}C_{\mathcal{Z}_{\delta}}^{1)}(\kappa^2
r,r;H_j,f_l) +\sum_{j=
i,\sigma(i)}\left(\sum_{l'=1}^2C_{\mathcal{Z}_{\delta}}^{n)}(\kappa^2
r,r;H_j,f_{l'}) -nC_{\mathcal{Z}_{\delta}}^{1)}(\kappa^2
r,r;H_j,f_l)\right)$$ \begin{equation}\le
C_{\mathcal{Z}_{\delta}}(\kappa^2 r,r;0,P_i).\label{eqn00-1-2}
\end{equation}
On the other hand, it is easy to see that
\begin{eqnarray}
C_{\mathcal{Z}_{\delta}}(\kappa^2 r,r;0,P_i)\le
C_{\mathcal{Z}_{\delta}}(r;0,P_i)\le
S_{\mathcal{Z}_{\delta}}(r,f_{1})+S_{\mathcal{Z}_{\delta}}(r,f_{2})+O(1).\label{eqn00-1-3}
\end{eqnarray}
Combining (\ref{eqn00-1-2}) and (\ref{eqn00-1-3}) implies
$$\sum_{\substack{j=1\\j\neq
i,\sigma(i)}}^{2n+3}C_{\mathcal{Z}_{\delta}}^{1)}(\kappa^2
r,r;H_j,f_l) +\sum_{j=
i,\sigma(i)}\left(\sum_{l'=1}^2C_{\mathcal{Z}_{\delta}}^{n)}(\kappa^2
r,r;H_j,f_{l'}) -nC_{\mathcal{Z}_{\delta}}^{1)}(\kappa^2
r,r;H_j,f_l)\right)$$ \begin{equation*}\leq
C_{\mathcal{Z}_{\delta}}(\kappa^2
r,r;0,P_i)\le\sum_{l'=1}^2S_{\mathcal{Z}_{\delta}}(r,f_{l'})+O(1).
\end{equation*}
Taking summation of the above inequality over $1\le i\le 2n+3$ and
noting that $\sigma$ is bijective, we have
$$(2n+1)\sum_{\substack{j=1}}^{2n+3}C_{\mathcal{Z}_{\delta}}^{1)}(\kappa^2
r,r;H_j,f_l)$$
$$+2\sum_{\substack{j=1}}^{2n+3}\left(\sum_{l'=1}^2C_{\mathcal{Z}_{\delta}}^{n)}(\kappa^2
r,r;H_j,f_{l'}) -nC_{\mathcal{Z}_{\delta}}^{1)}(\kappa^2
r,r;H_j,f_l)\right)$$
\begin{equation*}\le(2n+3)\sum_{l'=1}^2S_{\mathcal{Z}_{\delta}}(r,f_{l'})+O(1)
\end{equation*}
so that
$$\sum_{j=1}^{2n+3}C_{\mathcal{Z}_{\delta}}^{1)}(\kappa^2
r,r;H_j,f_l)+2\sum_{j=1}^{2n+3}\sum_{l'=1}^2C_{\mathcal{Z}_{\delta}}^{n)}(\kappa^2
r,r;H_j,f_{l'})$$
$$\le(2n+3)\sum_{l'=1}^2S_{\mathcal{Z}_{\delta}}(r,f_{l'})+O(1)$$
for $l=1,2$. Therefore, by noting that
$C_{\mathcal{Z}_{\delta}}^{1)}(\kappa^2
r,r;H_j,f_l)\geq\frac{1}{n}C_{\mathcal{Z}_{\delta}}^{n)}(\kappa^2
r,r;H_j,f_l)$, we have
\begin{eqnarray*}
\left(2+\frac{1}{2n}\right)\sum_{j=1}^{2n+3}\sum_{l=1}^2C_{\mathcal{Z}_{\delta}}^{n)}(\kappa^2
r,r;H_j,f_{l})\le(2n+3)\sum_{l=1}^2S_{\mathcal{Z}_{\delta}}(r,f_{l})+O(1).
\end{eqnarray*}
From Theorem \ref{thm2.1}, it follows that
\begin{eqnarray*}
&&\left(2+\frac{1}{2n}\right)(n+2)\sum_{l=1}^2S_{\mathcal{Z}_{\delta}}(r,f_{l})\\
&\le&\left(2+\frac{1}{2n}\right)\sum_{j=1}^{2n+3}\sum_{l=1}^2(C_{\mathcal{Z}_{\delta}}^{n)}(\kappa^2
r,r;H_j,f_{l}) +C_{\mathcal{Z}_{\delta}}^{n)}(\kappa^2
r;H_j,f_{l}))\\
&&+\left(2+\frac{1}{2n}\right)\sum_{j=1}^{2n+3}\sum_{l=1}^2\frac{\kappa^{2\omega}(\kappa^{-4\omega}-1)}{
r^{\omega}}n_{\mathcal{Z}_\delta}(\kappa^2 r,H_j,f_l)\\
&&+\left(2+\frac{1}{2n}\right)\sum_{l=1}^2R_{\mathcal{Z}_{\delta}}(r,f_l)\\
&\le&(2n+3)\sum_{l=1}^2S_{\mathcal{Z}_{\delta}}(r,f_{l})
+\left(2+\frac{1}{2n}\right)\sum_{l=1}^2R_{\mathcal{Z}_{\delta}}(r,f_l)\\
&&+
\left(2+\frac{1}{2n}\right)\sum_{j=1}^{2n+3}\sum_{l=1}^2C_{\mathcal{Z}_{\delta}}^{n)}(\kappa
r;H_j,f_{l})+O(1)\\
&&+\left(2+\frac{1}{2n}\right)\frac{\kappa^{-4\omega}-1}{
r^{\omega}\log
\kappa^{-1}}\sum_{j=1}^{2n+3}\sum_{l=1}^2N_{\mathcal{Z}_\delta}(\kappa
r,H_j,f_l).
\end{eqnarray*}
This implies immediately (\ref{eqn00-1-4}).\end{proof}

Then the proof of Theorem \ref{thm1.4} can be completed in terms of
(\ref{eqn00-1-4}) together with the methods in the proof of Theorem
\ref{thm1.1}. \qed

\

{\sl Proof of Theorem \ref{thm1.6}.}\ Suppose $f\not\equiv g$ and
set $f_1:=f$ and $f_2:=g$. Take $r'_m$ in $[r_m/2,r_m]$ outside the exceptional set
in Cartan's second main theorem
for $f_1$ and $f_2$. Below we always take $r=r'_m$.

Note that $P_i\not\equiv
0$ for $i=1,2,...,2n+3$. Similar to (\ref{eqn00-1-2}) and (\ref{eqn00-1-3}), we have, for $l=1,2$,
\begin{eqnarray*}
&&\sum_{\substack{j=1\\j\neq
i,\sigma(i)}}^{2n+3}(N^{1)}(r^{\sigma},H_j,f_l)-N^{1)}(2r,H_j,f_l))\\
&&+\sum_{j=
i,\sigma(i)}\sum_{l'=1}^2(N^{n)}(r^{\sigma},H_j,f_{l'})-N^{n)}(2r,H_j,f_{l'}))\\
&&-n\sum_{j=
i,\sigma(i)}(N^{1)}(r^{\sigma},H_j,f_{l})-N^{1)}(2r,H_j,f_{l}))\le N(r^{\sigma},0,P_i)\le T(r^{\sigma})+O(1).
\end{eqnarray*}
Taking summation of the above inequality over $1\le i\le 2n+3$ and
noting that $\sigma$ is bijective, we have
\begin{eqnarray*}
&&\sum_{j=1}^{2n+3}(N^{1)}(r^{\sigma},H_j,f_l)-N^{1)}(2r,H_j,f_l))\\
&&+2\sum_{j=1}^{2n+3}\sum_{l'=1}^2(N^{n)}(r^{\sigma},H_j,f_{l'})-N^{n)}(2r,H_j,f_{l'}))\le (2n+3)T(r^{\sigma})+O(1)
\end{eqnarray*}
for $l=1,2$. Hence,
\begin{eqnarray*}
&&(2+\frac{1}{2n})\sum_{j=1}^{2n+3}\sum_{l=1}^2N^{n)}(r^{\sigma},H_j,f_l)\le (2n+3)T(r^{\sigma})\\
&+&\frac{1}{2}\sum_{j=1}^{2n+3}\sum_{l=1}^2N^{1)}(2r,H_j,f_{l})+2\sum_{j=1}^{2n+3}\sum_{l=1}^2N^{n)}(2r,H_j,f_{l})+O(1)\\
&\le&(2n+3)T(r^{\sigma})+\frac{5}{2}(2n+3)T(2r)+O(1).
\end{eqnarray*}
By Cartan's second main theorem, we have
$$
(2+\frac{1}{2n})(n+2)T(r^{\sigma})\le (2n+3)T(r^{\sigma})+\frac{5}{2}(2n+3)T(2r)+O(\log^+rT(r^\sigma)).
$$
Note that $T(r^\sigma)\geq \frac{\sigma\log r}{\log 2r}
T(2r)+(1-\sigma)T(1)$. It implies that
$$
\sigma(\frac{3}{2}+\frac{1}{n})\le 5n+\frac{15}{2}
$$
and a contradiction is derived. \qed

\section{Proofs of Theorems \ref{thm1.5}, \ref{thm1.8} and \ref{thm1.10}}

To prove Theorems \ref{thm1.5}, \ref{thm1.8} and \ref{thm1.10}, we need some preparations.

Let $G$ be a torsion free abelian group and $A=(a_1,a_2,...,a_q)$ be
a $q$-tuple of elements $a_i$ in $G$. Let $q\geq r>s> 1.$ We say
that the $q$-tuple $A$\ has the property $(P_{r,s})$ if any $r$
elements $a_{l(1)},a_{l(2)},...,a_{l(r)}$ in $A$ satisfy the condition
that for any given $i_1,i_2,...,i_s\ (1\leq i_1<i_2<... <i_s\leq r),$
there exist $j_1,j_2,...,j_s\ (1\leq j_1<j_2<... <j_s\leq r)$ with
$\{i_1,i_2,...,i_s\} \neq \{j_1,j_2,...,j_s\}$ such that $a_{l(i_1)}a_{l(i_2)}
\cdots a_{l(i_s)}=a_{l(j_1)}a_{l(j_2)} \cdots a_{l(j_s)}.$ The following lemma can be found in \cite{Fujimoto4}.

\begin{lem}\label{lem4.1} Let $G$ be a torsion free abelian group and
$A=(a_1,a_2,...,a_q)$ be a $q$-tuple of elements $a_i$ in $G$. If $A$
has the property $(P_{r,s})$\ for some $r,s$\ with $q\geq r>s> 1,$
then there exist $i_1,i_2,...,i_{q-r+2}$ with $1\leq
i_1<i_2<...<i_{q-r+2}\leq q$ such that
$a_{i_1}=a_{i_2}=...=a_{i_{q-r+2}}$.
\end{lem}

We note that second main theorem holds for all $r>1$ except for a set of $r$ with finite
linear measure. In the proofs of Theorems \ref{thm1.5}, \ref{thm1.8} and \ref{thm1.10}, we use varied
second main theorems finite times, which causes an exceptional set of $r$ with finite
linear measure. We shall avoid this set in the proof.

\

{\sl Proof of Theorem \ref{thm1.5}.} \ Suppose $f\not\equiv g$ and set $f_1:=f$ and
$f_2:=g$. So $P_i\not\equiv 0$
for $i=1,2,...,2n+2$.

Assume that $f_1$ and $f_2$ satisfy

i) $ \min\{\nu_{\langle{\bf f}_1,{\bf a}_j\rangle},n+1\}=\min\{\nu_{\langle{\bf f}_2,{\bf a}_j\rangle},n+1\}$ on
$B(z,\varepsilon|z|)$ for $j=1,2,..., 2n+2$;

ii) $ f_1=f_2$ on $B(z,\varepsilon|z|)\cap\bigcup_{j=1}^{2n+2}
f^{-1}(H_j)$.

\noindent We can establish the following inequality, whose proof we
invite the reader to complete: For $z=re^{i\theta}$, $l=1,2$ and
$i=1,2,...,2n+2$, by using the same notations $\delta, \kappa$ and
$\mathcal{Z}_{\delta}=\mathcal{Z}_{\delta}(\theta)$ as in the proof of
Theorem \ref{thm1.1}, we have
$$\sum_{\substack{j=1\\j\neq
i,\sigma(i)}}^{2n+2}C_{\mathcal{Z}_{\delta}}^{1)}(\kappa^2
r,r;H_j,f_l)+\sum_{j=
i,\sigma(i)}C_{\mathcal{Z}_{\delta}}^{n+1)}(\kappa^2 r,r;H_j,f_{l})$$\begin{eqnarray}\le C_{\mathcal{Z}_{\delta}}(\kappa^2 r,r;0,P_{i})\le\sum_{l'=1}^2S_{\mathcal{Z}_{\delta}}(r,f_{l'})+O(1).\label{eqn01-1-1}\end{eqnarray}

Summing-up the above inequality over $1\le i\le 2n+2$ and noting
that $\sigma$ is bijective, we obtain
$$2n\sum_{j=1}^{2n+2}C_{\mathcal{Z}_{\delta}}^{1)}(\kappa^2
r,r;H_j,f_l)+2\sum_{j=1}^{2n+2}C_{\mathcal{Z}_{\delta}}^{n+1)}(\kappa^2
r,r;H_j,f_{l})$$$$
\le
(2n+2)\sum_{l'=1}^2S_{\mathcal{Z}_{\delta}}(r,f_{l'})+O(1),\
l=1,2.$$ By noting that $C_{\mathcal{Z}_{\delta}}^{1)}(\kappa^2 r,r;H_j,f_l)\ge
\frac{1}{n}C_{\mathcal{Z}_{\delta}}^{n)}(\kappa^2 r,r;H_j,f_l)$, we have, for $l=1,2$,
\begin{eqnarray*}
&&4\sum_{j=1}^{2n+2}C_{\mathcal{Z}_{\delta}}^{n)}(\kappa^2 r,r;H_j,f_l)
+2\sum_{j=1}^{2n+2}(C_{\mathcal{Z}_{\delta}}^{n+1)}(\kappa^2 r,r;H_j,f_{l})
-C_{\mathcal{Z}_{\delta}}^{n)}(\kappa^2 r,r;H_j,f_l))\\
&=&4\sum_{j=1}^{2n+2}C_{\mathcal{Z}_{\delta}}^{n)}(\kappa^2 r,r;H_j,f_l)
+2\sum_{j=1}^{2n+2}C_{\mathcal{Z}_{\delta},>n}^{1)}(\kappa^2 r,r;H_j,f_{l})
\end{eqnarray*}
\begin{eqnarray*}
\le (2n+2)\sum_{l'=1}^2S_{\mathcal{Z}_{\delta}}(r,f_{l'})+O(1),
\end{eqnarray*}
where $C_{\mathcal{Z}_{\delta},>n}^{1)}(\kappa^2 r,r;H_j,f_{l})$ is
the counting function in which we only consider the zeros of
$\langle{\bf f}_l,{\bf a}_j\rangle$ with multiplicity $>n$.

Using Theorem \ref{thm2.1} yields
\begin{eqnarray}
&&(4n+4)\sum_{l=1}^2S_{\mathcal{Z}_{\delta}}(r,f_{l})
+2\sum_{l=1}^2\sum_{j=1}^{2n+2}C_{\mathcal{Z}_{\delta},>n}^{1)}(\kappa^2 r,r;H_j,f_{l})\nonumber\\
&\le&(4n+4)\sum_{l=1}^2S_{\mathcal{Z}_{\delta}}(r,f_{l})
+4\sum_{l=1}^2\sum_{j=1}^{2n+2}C_{\mathcal{Z}_{\delta}}^{n)}(\kappa
r;H_j,f_{l})\nonumber\\
&&+4\frac{\kappa^{-4\omega}-1}{
r^{\omega}\log \kappa^{-1}}\sum_{j=1}^{2n+2}\sum_{l=1}^2N_{\mathcal{Z}_\delta}(\kappa
r,H_j,f_l)+4\sum_{l=1}^2R_{\mathcal{Z}_{\delta}}(r,f_{l})+O(1).\label{eqn01-1-2}
\end{eqnarray}
For simplicity, we set $N_1:=\frac{\kappa^{-4\omega}-1}{
r^{\omega}\log \kappa^{-1}}\sum_{l=1}^2\sum_{j=1}^{2n+2}N_{\mathcal{Z}_\delta}(\kappa
r,H_j,f_l)$. It implies that
\begin{eqnarray}
\sum_{l=1}^2\sum_{j=1}^{2n+2}C_{\mathcal{Z}_{\delta},>n}^{1)}(\kappa^2
r,r;H_j,f_{l})&\le&
2\sum_{l=1}^2\sum_{j=1}^{2n+2}C_{\mathcal{Z}_{\delta}}^{n)}(\kappa
r;H_j,f_{l})+2N_1\nonumber\\
&&+2\sum_{l=1}^2R_{\mathcal{Z}_{\delta}}(r,f_{l})+O(1).\label{eqn00-2-2}
\end{eqnarray}

Assume that ${\bf f}_1=({f_1}_0,{f_1}_1,...,{f_1}_n)$ and
${\bf f}_2=({f_2}_0,{f_2}_1,...,{f_2}_n)$ are the reduced representations of $f_1$
and $f_2$, respectively.

Denote by $\mathcal{M}^*$ the abelian multiplicative group of all
nonzero meromorphic functions on ${\mathbb{C}}$ and
${\mathbb{C}}^*={\mathbb{C}}\setminus\{0\}$. Then the multiplicative
group $\mathcal{M}^*/{\mathbb{C}}^*$ is a torsion free abelian
group.

\

{\bf Step 1.}  In this step, we will show what has the property
$(P_{r,s}).$

Define $h_i=\frac{\langle{\bf f}_1,{\bf a}_i\rangle}{\langle{\bf f}_2,{\bf a}_i\rangle}$, $i=1,2,...,2n+2$.
Although each $h_i$ is dependent on the choice of reduced
representations of $f_1$ and $f_2,$ the ratio $h_p/h_q=\langle{\bf f}_1,{\bf a}_p\rangle
/\langle{\bf f}_2,{\bf a}_p\rangle\cdot \langle{\bf f}_2,{\bf a}_q\rangle/\langle{\bf f}_1,{\bf a}_q\rangle$\ is uniquely determined
independent of any choice of reduced representations of $f_1$ and
$f_2$. By the definition, we have
$$
\sum_{m=0}^na_{im}{f_1}_m-h_i\sum_{m=0}^na_{im}{f_2}_m=0\ \
(i=1,2,...,2n+2).
$$
Therefore
$$
\mbox{det}(a_{i0},a_{i1},...,a_{in},a_{i0}h_i,a_{i1}h_i,...,a_{in}h_i;\ \ 1\leq
i\leq 2n+2)=0.
$$
Let $\mathcal{I}$ be the set of all combinations
$I=\{i_1,i_2,...,i_{n+1}\}$ with $1\leq i_1<i_2<... <i_{n+1}\leq 2n+2$
of indices $1,2,...,2n+2$. For any $I=\{i_1,i_2,...,i_{n+1}\}\in
\mathcal{I}$, define
$$
h_I:=h_{i_1}h_{i_2}\cdots h_{i_{n+1}}
$$
and
\begin{eqnarray*}
A_I:=&&(-1)^{(n+1)(n+2)/2+i_1+i_2+...+i_{n+1}}\mbox{det}(a_{{i_r}l};1\leq
r\leq {n+1},0\leq l\leq n)\\
&\times & \mbox{det}(a_{{j_s}l};1\leq s\leq {n+1},0\leq l\leq n),
\end{eqnarray*}
where $J=\{j_1,j_2,...,j_{n+1}\}\in \mathcal{I}$ such that $I\cup
J=\{1,2,...,2n+2\}$. Then we have
$$
\sum_{I\in \mathcal{I}}A_Ih_I=0,
$$
where $A_I\neq 0$  and $A_I/A_J\in {\mathbb{C}}^* $ for any $I,J\in
\mathcal{I}$.

Now we show that, for each $I\in \mathcal{I},$ there exists $J\in
\mathcal{I}$ with $I\neq J$ such that $\frac{h_I}{h_J}\in
{\mathbb{C}}^*$.

Let $I_0=I$. Denote by $\tau$ the minimal number satisfying the
following:

There exist $\tau$ elements $I_1,I_2,...,I_{\tau}\in\mathcal{I}\
\backslash{I_0}$ and $\tau$ nonzero constants $b_i$ such that
$h_{I_0}=\sum_{i=1}^{\tau}b_ih_{I_i}$. Let $b_0=-1$. Then
\begin{eqnarray}
\sum_{i=0}^{\tau}b_ih_{I_i}=0.\label{eq45}
\end{eqnarray}
Since $h_{I_0}\not\equiv 0$ and by the minimality of $\tau$, it
follows that the family $\{h_{I_1},h_{I_2},...,h_{I_{\tau}}\}$ is
linearly independent over ${\mathbb{C}}$.

Now it suffices to show that $\tau=1$. Assume that $\tau\geq 2$.

Set $I=\bigcap_{i=0}^{\tau}I_i$, $I'_i=I_i\backslash I\neq \emptyset
\ (0\leq i\leq \tau)$ and $\tilde I=\bigcup _{i=0}^{\tau}I'_i$,
$I'=\bigcap_{i=1}^{\tau}I'_i$, $I''_i=I'_i\backslash I'\neq
\emptyset\ (1\leq i\leq \tau)$. We have
\begin{eqnarray}
\frac{h_{I'_0}}{h_{I'}}=\sum_{i=1}^{\tau}b_ih_{I''_i}.\label{eq46}
\end{eqnarray}

For $\tau\geq 2$, we can construct a holomorphic curve
$h:{\mathbb{C}}\rightarrow {\mathcal{P}}^{\tau-1}({\mathbb{C}})$ with
a reduced representation
$$
{\bf h}=\left(b_1\tilde hh_{I''_1},b_2\tilde hh_{I''_2},..., b_{\tau}\tilde
hh_{I''_{\tau}}\right),
$$
where $\tilde h$ is  holomorphic on ${\mathbb{C}}$ such that
$(b_1\tilde hh_{I''_1},b_2\tilde hh_{I''_2},..., b_{\tau}\tilde
hh_{I''_{\tau}})$
becomes a reduced representation. We have $\nu_{\tilde
h}=\sum_{i\in{\bigcup_{1\leq j\leq \tau}I''_j}}\nu_{h_i}^\infty$.

It is easy to see that the holomorphic curve $h$ is linearly
non-degenerate over ${\mathbb{C}}$ by (\ref{eq45}). Consider the
hyperplanes $\tilde{H}_i=\{w_i=0\},i=1,2,...,\tau$, and
$\tilde{H}_{\tau+1}=\{w_1+w_2+...+w_{\tau}=0\}$. We have
$$
\nu_{\langle{\bf h},\tilde{\bf a}_i\rangle}=\nu_{\tilde h h_{I''_i}}\ \ i=1,2,...,\tau,
$$
and
$$
\nu_{\langle{\bf h},\tilde{\bf a}_{\tau+1}\rangle}=\nu_{\tilde h \frac{h_{I'_0}}{h_{I'}}}
$$
by (\ref{eq46}).

Now, we estimate $\nu^{1)}_{\tilde h h_{I''_i}}$ and
$\nu^{1)}_{\tilde h \frac{h_{I'_0}}{h_{I'}}}$.

Define $I''=\bigcup_{i=1}^{\tau}I''_i.$  Then
$$
\nu^{1)}_{\tilde{h}h_{I''_{i}}}= \nu^{1)}_{h_{I_i ''}}+\nu^{1)}_{
\frac{1}{h_ {{I''}\backslash {I''_i}}}}\ \mbox{ and} \
\nu^{1)}_{\tilde {h}\frac{h_{I'_0} }{h_{I'}}}=\nu^{1)}_{h_{I_0
'}}+\nu^{1)}_{ \frac{1}{h_ {{(I'' \bigcup I')}\backslash {I'_0}}}}.
$$

For each $J\subset\{1,2,...,2n+2\}$, put
$J^c=\{1,2,...,2n+2\}\backslash J$. It is easy to see that
$$
I''_i\subset I_i\ \ \mbox{and}\ \ I''\backslash I''_i\subset I^c_i\
(1\leq i\leq \tau),
$$
$$
I'_0\subset I_0\ \ \mbox{and}\ \ (I''\cup I')\backslash I'_0=\tilde
I \backslash(I_0\backslash I)=(\tilde I\cup I)\backslash I_0\subset
I_0^c.
$$
Hence
$$
\nu^{1)}_{\tilde{h}h_{I''_{i}}}\leq \nu^{1)}_{
h_{I_i}}(r)+\nu^{1)}_{\frac{1}{h_{I^c_i}}} \ \mbox{and}\
\nu^{1)}_{\tilde {h}\frac{h_{I'_0} }{h_{I'}}}\leq \nu^{1)}_{
h_{I_0}}+\nu^{1)}_{\frac{1}{h_{I^c_0}}}.
$$

Now we consider the zeros and poles of $h_i$. If $z\in
B(z,\varepsilon|z|)$ is a zero(a pole) of $h_i$, then
$\nu_{\langle{\bf f}_1,{\bf a}_i\rangle}(z)\neq \nu_{\langle{\bf f}_2,{\bf a}_i\rangle}(z)$. Hence
$\nu_{\langle{\bf f}_l,{\bf a}_i\rangle}(z)>n$. It follows from (\ref{eqn00-2-2}) that
\begin{eqnarray}
C_{\mathcal{Z}_{\delta}}^{1)}(\kappa^2 r,r;0,\tilde{h}h_{I''_{i}})&\le&
C_{\mathcal{Z}_{\delta}}^{1)}(\kappa^2 r,r;0,h_{I_i})+C_{\mathcal{Z}_{\delta}}^{1)}(\kappa^2 r,r;\infty,h_{I^c_i})\nonumber\\
&\le&\sum_{l=1}^2\sum_{j=1}^{2n+2}C_{\mathcal{Z}_{\delta},>n}^{1)}(\kappa^2
r,r;H_j,f_{l})\nonumber\\
&\le&
2\sum_{l=1}^2\sum_{j=1}^{2n+2}C_{\mathcal{Z}_{\delta}}^{n)}(\kappa
r;H_j,f_{l})+2N_1+2\sum_{l=1}^2R_{\mathcal{Z}_{\delta}}(r,f_{l})+O(1)\label{eqn00-2-3}
\end{eqnarray}
and
\begin{eqnarray}
C_{\mathcal{Z}_{\delta}}^{1)}(\kappa^2 r,r;0,\tilde
{h}\frac{h_{I'_0} }{h_{I'}})&\le& C_{\mathcal{Z}_{\delta}}^{1)}(\kappa^2 r,r;0,h_{I_0})+C_{\mathcal{Z}_{\delta}}^{1)}(\kappa^2 r,r;\infty,h_{I^c_0})\nonumber\\
&\le&\sum_{l=1}^2\sum_{j=1}^{2n+2}C_{\mathcal{Z}_{\delta},>n}^{1)}(\kappa^2
r,r;H_j,f_{l})\nonumber\\
&\le&
2\sum_{l=1}^2\sum_{j=1}^{2n+2}C_{\mathcal{Z}_{\delta}}^{n)}(\kappa
r;H_j,f_{l})+2N_1+2\sum_{l=1}^2R_{\mathcal{Z}_{\delta}}(r,f_{l})+O(1).\label{eqn00-2-4}
\end{eqnarray}
Using Theorem \ref{thm2.1} for $h$ and
$\{\tilde{H}_i\}_{i=1}^{\tau+1}$ deduces that
\begin{eqnarray}
S_{\mathcal{Z}_{\delta}}(r,h)&\le&\sum_{i=1}^{\tau+1}C^{\tau-1)}_{\mathcal{Z}_{\delta}}(r;\tilde{H}_i,h)+R_{\mathcal{Z}_{\delta}}(r,h)\nonumber\\
&\le&(\tau-1)\sum_{i=1}^{\tau}C_{\mathcal{Z}_{\delta}}^{1)}(\kappa^2
r,r;0,\tilde{h}h_{I''_{i}})+(\tau-1)C_{\mathcal{Z}_{\delta}}^{1)}(\kappa^2
r,r;0,\tilde
{h}\frac{h_{I'_0} }{h_{I'}})\nonumber\\
&+&\sum_{i=1}^{\tau+1}C^{\tau-1)}_{\mathcal{Z}_{\delta}}(\kappa r;\tilde{H}_i,h)+\frac{\kappa^{-4\omega}-1}{
r^{\omega}\log \kappa^{-1}}\sum_{i=1}^{\tau+1}N_{\mathcal{Z}_\delta}(\kappa
r,\tilde{H}_i,h)+R_{\mathcal{Z}_{\delta}}(r,h)\nonumber
\end{eqnarray}
\begin{eqnarray}
&\le&2(\tau^2-1)\left(\sum_{l=1}^2\sum_{j=1}^{2n+2}C_{\mathcal{Z}_{\delta}}^{n)}(\kappa
r;H_j,f_{l})+N_1+\sum_{l=1}^2R_{\mathcal{Z}_{\delta}}(r,f_{l})\right)+O(1)\nonumber\\
&+&\sum_{i=1}^{\tau+1}C^{\tau-1)}_{\mathcal{Z}_{\delta}}(\kappa
r;\tilde{H}_i,h)+\frac{\kappa^{-4\omega}-1}{
r^{\omega}\log \kappa^{-1}}\sum_{i=1}^{\tau+1}N_{\mathcal{Z}_\delta}(\kappa
r,\tilde{H}_i,h)+R_{\mathcal{Z}_{\delta}}(r,h).\label{eqn00-2-5}
\end{eqnarray}

On the other hand, we have
\begin{eqnarray*}
S_{\mathcal{Z}_{\delta}}(r,h)&\geq&
S_{\mathcal{Z}_{\delta}}\left(r,\frac{\langle{\bf h},\tilde{\bf
a}_1\rangle}{\langle{\bf h},\tilde{\bf a}_2\rangle}\right)+O(1)
=S_{\mathcal{Z}_{\delta}}\left(r,\frac{h_{I''_1}}{h_{I''_2}}\right)+O(1)\\
&=&S_{\mathcal{Z}_{\delta}}\left(r,\frac{h_{I_1}}{h_{I_2}}\right)+O(1)\geq
C_{\mathcal{Z}_{\delta}}^{1)}(r;1,\frac{h_{I_1}}{h_{I_2}})+O(1),
\end{eqnarray*}
\begin{eqnarray*}
S_{\mathcal{Z}_{\delta}}(r,h)&\geq&
S_{\mathcal{Z}_{\delta}}\left(r,\frac{\langle{\bf h},\tilde{\bf a}_2\rangle}{\langle{\bf h},\tilde{\bf a}_{\tau+1}\rangle}\right)+O(1)=S_{\mathcal{Z}_{\delta}}\left(r,\frac{h_{I''_2}}{h_{I''_0}}\right)+O(1)\\
&=&S_{\mathcal{Z}_{\delta}}\left(r,\frac{h_{I_2}}{h_{I_0}}\right)+O(1)\geq
C_{\mathcal{Z}_{\delta}}^{1)}(r;1,\frac{h_{I_2}}{h_{I_0}})+O(1)
\end{eqnarray*}
and
\begin{eqnarray*}
S_{\mathcal{Z}_{\delta}}(r,h)&\geq&
S_{\mathcal{Z}_{\delta}}\left(r,\frac{\langle{\bf h},\tilde{\bf a}_{\tau+1}\rangle}{\langle{\bf h},\tilde{\bf a}_1\rangle}\right)+O(1)=S_{\mathcal{Z}_{\delta}}\left(r,\frac{h_{I''_0}}{h_{I''_1}}\right)+O(1)\\
&=&S_{\mathcal{Z}_{\delta}}\left(r,\frac{h_{I_0}}{h_{I_1}}\right)+O(1)\geq
C_{\mathcal{Z}_{\delta}}^{1)}(r;1,\frac{h_{I_0}}{h_{I_1}})+O(1).
\end{eqnarray*}

Since $ f_1=f_2$ on $B(z,\varepsilon|z|)\cap\bigcup_{j=1}^{2n+2}
f^{-1}(H_j)$. That is to say, $\frac{h_I}{h_J}=1$ on the set
$B(z,\varepsilon|z|)\cap\bigcup_{j\in((I\cup J)\backslash(I\cap
J))^c}f^{-1}(H_j)$. By
$$
((I_1\cup I_2)\backslash(I_1\cap I_2))^c\cup((I_2\cup
I_0)\backslash(I_2\cap I_0))^c\cup ((I_0\cup I_1)\backslash(I_0\cap
I_1))^c=\{1,2,...,2n+2\},
$$
it implies that
\begin{eqnarray}
3S_{\mathcal{Z}_{\delta}}(r,h)&\geq&C_{\mathcal{Z}_{\delta}}^{1)}(r;1,\frac{h_{I_1}}{h_{I_2}})+C_{\mathcal{Z}_{\delta}}^{1)}(r;1,\frac{h_{I_2}}{h_{I_0}})+C_{\mathcal{Z}_{\delta}}^{1)}(r;1,\frac{h_{I_0}}{h_{I_1}})+O(1)\nonumber\\
&\ge&\sum_{j=1}^{2n+2}C_{\mathcal{Z}_{\delta}}^{1)}(\kappa^2
r,r;H_j,f_{l})+O(1)\nonumber\\
&\ge &
\frac{1}{n}\sum_{j=1}^{2n+2}C_{\mathcal{Z}_{\delta}}^{n)}(\kappa^2
r,r;H_j,f_{l})+O(1),\ \ l=1,2.\label{eqn00-2-6}
\end{eqnarray}
Using Theorem \ref{thm2.1} again, we have
\begin{eqnarray}
\frac{n+1}{n}\sum_{l=1}^2S_{\mathcal{Z}_{\delta}}(r,f_l)&\le&\frac{1}{n}\sum_{l=1}^2\sum_{j=1}^{2n+2}C_{\mathcal{Z}_{\delta}}^{n)}(\kappa^2
r,r;H_j,f_{l})+\frac{1}{n}\sum_{l=1}^2\sum_{j=1}^{2n+2}C_{\mathcal{Z}_{\delta}}^{n)}(\kappa^2
r;H_j,f_{l})\nonumber\\
&&+\frac{1}{n}N_1+\frac{1}{n}\sum_{l=1}^2R_{\mathcal{Z}_{\delta}}(r,f_l).\label{eqn00-2-7}
\end{eqnarray}
Combining (\ref{eqn00-2-6}) and (\ref{eqn00-2-7}) yields
\begin{eqnarray}
\frac{n+1}{n}\sum_{l=1}^2S_{\mathcal{Z}_{\delta}}(r,f_l)&\le&
6S_{\mathcal{Z}_{\delta}}(r,h)+\frac{1}{n}\sum_{l=1}^{2}\sum_{j=1}^{2n+2}C_{\mathcal{Z}_{\delta}}^{n)}(\kappa
r;H_j,f_{l})\nonumber\\
&&+\frac{1}{n}N_1+\frac{1}{n}\sum_{l=1}^{2}R_{\mathcal{Z}_{\delta}}(r,f_l)
.\label{eqn00-2-8}
\end{eqnarray}
From (\ref{eqn00-2-5}) and (\ref{eqn00-2-8}), it follows that
\begin{eqnarray}
&&\frac{n+1}{6n}\sum_{l=1}^2S_{\mathcal{Z}_{\delta}}(r,f_l)\nonumber\\&\le& \left(2(\tau^2-1)+\frac{1}{6n}\right)\left(\sum_{l=1}^{2}\sum_{j=1}^{2n+2}C_{\mathcal{Z}_{\delta}}^{n)}(\kappa r;H_j,f_{l})+N_1+\sum_{l=1}^2R_{\mathcal{Z}_{\delta}}(r,f_l)\right)\label{eqn00-2-9}\\
&&+\sum_{i=1}^{\tau+1}C^{\tau-1)}_{\mathcal{Z}_{\delta}}(\kappa
r;\tilde{H}_i,h)+\frac{\kappa^{-4\omega}-1}{
r^{\omega}\log \kappa^{-1}}\sum_{i=1}^{\tau+1}N_{\mathcal{Z}_\delta}(\kappa
r,\tilde{H}_i,h)+R_{\mathcal{Z}_{\delta}}(r,h).\nonumber
\end{eqnarray}

In order to treat terms $C^{\tau-1)}_{\mathcal{Z}_{\delta}}(\kappa
r;\tilde{H}_i,h)$, $N_{\mathcal{Z}_\delta}(\kappa
r,\tilde{H}_i,h)$ and $R_{\mathcal{Z}_{\delta}}(r,h)$, we need to
compare $T(r,h)$ with $T(r)=\sum_{l=1}^2T(r,f_l)$. Similar to
(\ref{eqn00-2-3}) and (\ref{eqn00-2-4}), we have
\begin{eqnarray*}
&&N^{1)}(r,0,\tilde{h}h_{I''_{i}})\le N^{1)}(r,0,h_{I_i})+N^{1)}(r,\infty,h_{I^c_i})\\
&\le&\sum_{j=1}^{2n+2}N^{1)}(r,H_j,f_{1})+\sum_{j=1}^{2n+2}N^{1)}(r,H_j,f_{2})\le
(2n+2)T(r)+O(1)
\end{eqnarray*}
and
\begin{eqnarray*}
&&N^{1)}(r,0,\tilde
{h}\frac{h_{I'_0} }{h_{I'}})\le N^{1)}(r,0,h_{I_0})+N^{1)}(r,\infty,h_{I^c_0})\\
&\le&\sum_{j=1}^{2n+2}N^{1)}(r,H_j,f_{1})+\sum_{j=1}^{2n+2}N^{1)}(r,H_j,f_{2})\le
(2n+2)T(r)+O(1).
\end{eqnarray*}
Using Cartan's second main theorem for $h$ and
$\{\tilde{H}_j\}_{j=1}^{\tau+1}$, we have
\begin{eqnarray*}
T(r,h) &\leq&(\tau-1)\sum_{i=1}^{\tau}
N^{1)}(r,0,\tilde{h}h_{I''_{i}})+(\tau-1)N^{1)}(r,0,\tilde
{h}\frac{h_{I'_0} }{h_{I'}})+o(T(r,h))\\
&\le&(\tau^2-1)(2n+2)T(r)+o(T(r,h)).
\end{eqnarray*}
Hence, for $r$ big enough,
\begin{eqnarray}
T(r,h)\le 2(\tau^2-1)(2n+2)T(r).\label{eqn00-2-10}
\end{eqnarray}
Thus, by Lemma \ref{lem2.1} and (\ref{eqn00-2-10}), we have
\begin{eqnarray*}
C^{\tau-1)}_{\mathcal{Z}_{\delta}}(\kappa r;\tilde{H}_i,h)&\le&\frac{2\omega}{\kappa^{\omega}}\frac{T(\kappa r,h)}{r^{\omega}}+\omega^2\int_{1}^{\kappa r}\frac{T(t,h)}{t^{\omega+1}}{\rm d}t+O(1)\\
&\le&2(\tau^2-1)(2n+2)\left(\frac{2\omega}{\kappa^{\omega}}\frac{T(\kappa
r)}{r^{\omega}}+\omega^2\int_{1}^{\kappa
r}\frac{T(t)}{t^{\omega+1}}{\rm d}t\right)+O(1),
\end{eqnarray*}
$$
N_{\mathcal{Z}_\delta}(\kappa
r,\tilde{H}_i,h)<e\kappa^s T(r,h)\le 2(\tau^2-1)(2n+2)e\kappa^s T(r)
$$
and
$$
R_{\mathcal{Z}_{\delta}}(r,h)\le K\omega(\log^+T(r,h)+\log^+r+1)\le
K'\omega(\log^+T(r)+\log^+r+1).
$$

Repeating the similar argument in the proof of Theorem \ref{thm1.1}, we can
find a sequence $\{z_m\}$ such that if the conditions in Theorem
\ref{thm1.5} holds for $X=B(\{z_m\};\{\varepsilon_m\})$, then we can
derive a contradiction from (\ref{eqn00-2-9}).

Hence $\tau=1$, i.e., $\frac{h_{I_0}}{h_{I_1}}=b_1\in
{\mathbb{C}}^*$. We have proved that, for each $I\in \mathcal{I}$,
there is $J\in \mathcal{I}\setminus\{I\}$ such that $h_I/h_J\in
{\mathbb{C}}^*$. We conclude that the family
$\{[h_1],[h_2],...,[h_q]\}$ has the property $(P_{2n+2,n+1})$ in the
torsion free abelian group $\mathcal{M}^*/{\mathbb{C}}^*$ by the
definition.

\

{\bf Step 2.} By Lemma \ref{lem4.1}, there exist $2$ elements, we
may assume that they are $[h_1]$ and $[h_{2}]$, such that
$[h_1]=[h_{2}]$. Then $\frac{h_1}{h_2}\in {\mathbb{C}}^*$, i.e.,
$\frac{\langle{\bf f}_1,{\bf a}_1\rangle}{\langle{\bf f}_2,{\bf
a}_1\rangle}=c\frac{\langle{\bf f}_1,{\bf a}_2\rangle}{\langle{\bf
f}_2,{\bf a}_2\rangle}$, where $c\in {\mathbb{C}}^*$.

If $n=1$, we hold Theorem \ref{thm1.8}.

For $n\ge 2$,
$c=1$ by $f_1=f_2$ on $X\cap \bigcup_{j=1}^{2n+2}f^{-1}(H_j)$.

According to the choices of $\{r_m\}$, $\{\delta_m\}$, $\{\kappa_m\}$, $\{\omega_m\}$ and $\{k_m\}$, we have
\begin{eqnarray}
\sum_{l=1}^2\sum_{j=1}^{2n+2}C_{\mathcal{Z}_{\delta_m}^{k_m}}^{n)}(\kappa_mr_m;H_j,f_{l})=o(1)\frac{T(r_m)}{r_m^{\omega_m}},\label{eqnN1}
\end{eqnarray}
\begin{eqnarray}
\frac{\kappa_m^{-4\omega_m}-1}{
r_m^{\omega_m}\log \kappa_m^{-1}}\sum_{l=1}^2\sum_{j=1}^{2n+2}N_{\mathcal{Z}_{\delta_m}^{k_m}}(\kappa_m
r_m,H_j,f_l)=o(1)\frac{T(r_m)}{r_m^{\omega_m}},\label{eqnN2}
\end{eqnarray}
\begin{eqnarray}
\sum_{l=1}^2R_{\mathcal{Z}_{\delta_m}^{k_m}}(r_m,f_{l})=o(1)\frac{T(r_m)}{r_m^{\omega_m}}\label{eqnN3}
\end{eqnarray}
and
\begin{eqnarray}
\sum_{l=1}^2S_{\mathcal{Z}_{\delta_m}^{k_m}}(r_m,f_{l})\ge \frac{\sqrt{2}}{10}\frac{T(r_m)}{r_m^{\omega_m}}.\label{eqnN4}
\end{eqnarray}
Hence, by (\ref{eqn00-2-2}),
$$
\sum_{l=1}^2\sum_{j=1}^{2n+2}C_{\mathcal{Z}_{\delta_m}^{k_m},>n}^{1)}(\kappa_m^2
r_m,r_m;H_j,f_{l})=o(1)\frac{T(r_m)}{r_m^{\omega_m}},
$$
i.e.,
\begin{eqnarray}
C_{\mathcal{Z}_{\delta_m}^{k_m},>n}^{1)}(\kappa_m^2
r_m,r_m;H_j,f_{l})=o(1)\frac{T(r_m)}{r_m^{\omega_m}}\label{eqnN5}
\end{eqnarray}
for $l=1,2$ and $j=1,2,...,2n+2$. According to Si's method, in view of (\ref{eqnN1}), (\ref{eqnN2}), (\ref{eqnN3})
and (\ref{eqnN5}), the both sides of (\ref{eqn01-1-2}) are equal up to $o(1)\frac{T(r_m)}{r_m^{\omega_m}}$. That
means those inequalities used to derive (\ref{eqn01-1-2}) become equalities up to $o(1)\frac{T(r_m)}{r_m^{\omega_m}}$.
Hence, we have the following equations:
\begin{eqnarray}
nC_{\mathcal{Z}_{\delta_m}^{k_m}}^{1)}(\kappa_m^2
r_m,r_m;H_j,f_l)= C_{\mathcal{Z}_{\delta_m}^{k_m}}^{n)}(\kappa_m^2
r_m,r_m;H_j,f_l)+o(1)\frac{T(r_m)}{r_m^{\omega_m}},\label{eqnN6}
\end{eqnarray}
\begin{eqnarray}
(n+1)S_{\mathcal{Z}_{\delta_m}^{k_m}}(r_m,f_{l})=\sum_{j=1}^{2n+2}C_{\mathcal{Z}_{\delta_m}^{k_m}}^{n)}(\kappa_m^2
r_m,r_m;H_j,f_l)+o(1)\frac{T(r_m)}{r_m^{\omega_m}}\label{eqnN7}
\end{eqnarray}
and
$$\sum_{\substack{j=1\\j\neq
i,\sigma(i)}}^{2n+2}C_{\mathcal{Z}_{\delta_m}^{k_m}}^{1)}(\kappa_m^2
r_m,r_m;H_j,f_l)+\sum_{j=
i,\sigma(i)}C_{\mathcal{Z}_{\delta_m}^{k_m}}^{n+1)}(\kappa_m^2 r_m,r_m;H_j,f_{l})+o(1)\frac{T(r_m)}{r_m^{\omega_m}}$$
\begin{eqnarray}= C_{\mathcal{Z}_{\delta_m}^{k_m}}(\kappa_m^2 r_m,r_m;0,P_{i})=\sum_{l'=1}^2S_{\mathcal{Z}_{\delta_m}^{k_m}}(r_m,f_{l'})+o(1)\frac{T(r_m)}{r_m^{\omega_m}},\label{eqnN8}\end{eqnarray}
where $l=1,2$ and $i,j=1,2,...,2n+2$. Therefore in view of (\ref{eqnN5}), we have
\begin{eqnarray}
C_{\mathcal{Z}_{\delta_m}^{k_m}}(\kappa_m^2
r_m,r_m;H_j,f_l)&=& C_{\mathcal{Z}_{\delta_m}^{k_m}}^{n+1)}(\kappa_m^2
r_m,r_m;H_j,f_l)+o(1)\frac{T(r_m)}{r_m^{\omega_m}}\nonumber\\
&=& C_{\mathcal{Z}_{\delta_m}^{k_m}}^{n)}(\kappa_m^2
r_m,r_m;H_j,f_l)+o(1)\frac{T(r_m)}{r_m^{\omega_m}}\label{eqnN9}
\end{eqnarray}
for $l=1,2$ and $j=1,2,...,2n+2$. Combining (\ref{eqnN7}) and (\ref{eqnN9}) yields
\begin{eqnarray*}
(n+1)S_{\mathcal{Z}_{\delta_m}^{k_m}}(r_m,f_{1})&=&\sum_{j=1}^{2n+2}C_{\mathcal{Z}_{\delta_m}^{k_m}}^{n)}(\kappa_m^2
r_m,r_m;H_j,f_1)+o(1)\frac{T(r_m)}{r_m^{\omega_m}}\\
&=&\sum_{j=1}^{2n+2}C_{\mathcal{Z}_{\delta_m}^{k_m}}^{n)}(\kappa_m^2
r_m,r_m;H_j,f_2)+o(1)\frac{T(r_m)}{r_m^{\omega_m}}\\
&=&(n+1)S_{\mathcal{Z}_{\delta_m}^{k_m}}(r_m,f_{2})+o(1)\frac{T(r_m)}{r_m^{\omega_m}},
\end{eqnarray*}
i.e.,
\begin{eqnarray}
S_{\mathcal{Z}_{\delta_m}^{k_m}}(r_m,f_{1})=S_{\mathcal{Z}_{\delta_m}^{k_m}}(r_m,f_{2})+o(1)\frac{T(r_m)}{r_m^{\omega_m}}.\label{eqnN10}
\end{eqnarray}
It follows from (\ref{eqnN6}), (\ref{eqnN7}), (\ref{eqnN8}) and (\ref{eqnN9}) that
$$\frac{1}{n}\sum_{\substack{j=1\\j\neq
i,\sigma(i)}}^{2n+2}C_{\mathcal{Z}_{\delta_m}^{k_m}}^{n)}(\kappa_m^2
r_m,r_m;H_j,f_l)+\sum_{j=
i,\sigma(i)}C_{\mathcal{Z}_{\delta_m}^{k_m}}^{n)}(\kappa_m^2 r_m,r_m;H_j,f_{l})+o(1)\frac{T(r_m)}{r_m^{\omega_m}}$$
$$=\sum_{l'=1}^2S_{\mathcal{Z}_{\delta_m}^{k_m}}(r_m,f_{l'})+o(1)\frac{T(r_m)}{r_m^{\omega_m}}$$
and
$$\frac{1}{n}\sum_{j=1}^{2n+2}C_{\mathcal{Z}_{\delta_m}^{k_m}}^{n)}(\kappa_m^2
r_m,r_m;H_j,f_l)+\frac{n-1}{n}\sum_{j=
i,\sigma(i)}C_{\mathcal{Z}_{\delta_m}^{k_m}}^{n)}(\kappa_m^2 r_m,r_m;H_j,f_{l})$$
$$=\frac{n+1}{n}S_{\mathcal{Z}_{\delta_m}^{k_m}}(r_m,f_{l})+\frac{n-1}{n}\sum_{j=
i,\sigma(i)}C_{\mathcal{Z}_{\delta_m}^{k_m}}^{n)}(\kappa_m^2 r_m,r_m;H_j,f_{l})+o(1)\frac{T(r_m)}{r_m^{\omega_m}}.$$
Combining the above two equations and (\ref{eqnN10}), we obtain, for $l=1,2$,
\begin{eqnarray}
S_{\mathcal{Z}_{\delta_m}^{k_m}}(r_m,f_{l})=\sum_{j=
i,\sigma(i)}C_{\mathcal{Z}_{\delta_m}^{k_m}}^{n)}(\kappa_m^2 r_m,r_m;H_j,f_{l})+o(1)\frac{T(r_m)}{r_m^{\omega_m}}.\label{eqnN11}
\end{eqnarray}

By $h_1\equiv h_2$, we have, for $i=1,2$,
$$
\nu_{\langle{\bf f}_1,{\bf
a}_i\rangle}(z)=\nu_{\langle{\bf f}_2,{\bf
a}_i\rangle}(z),\ z\in B(\{z_m\};\{\varepsilon_m\})
$$
and
$
\langle{\bf f}_1,{\bf
a}_1\rangle=\frac{\langle{\bf f}_1,{\bf
a}_2\rangle}{\langle{\bf f}_2,{\bf
a}_2\rangle}\langle{\bf f}_2,{\bf
a}_1\rangle
$,
where $\frac{\langle{\bf f}_1,{\bf
a}_2\rangle}{\langle{\bf f}_2,{\bf
a}_2\rangle}$ has no zeros and poles on $B(\{z_m\};\{\varepsilon_m\})$. Now, we consider
\begin{eqnarray*}
P_1&=&\langle{\bf f}_1,{\bf a}_1\rangle\langle{\bf f}_2,{\bf a}_{n+1}\rangle-\langle{\bf f}_1,{\bf a}_{n+1}\rangle\langle{\bf f}_2,{\bf a}_1\rangle\\
&=&\langle{\bf f}_2,{\bf a}_1\rangle\left(\frac{\langle{\bf f}_1,{\bf
a}_2\rangle}{\langle{\bf f}_2,{\bf
a}_2\rangle}\langle{\bf f}_2,{\bf a}_{n+1}\rangle-\langle{\bf f}_1,{\bf a}_{n+1}\rangle\right)\\
&=&\langle{\bf f}_1,{\bf a}_1\rangle\left(\langle{\bf f}_2,{\bf a}_{n+1}\rangle-\frac{\langle{\bf f}_2,{\bf
a}_2\rangle}{\langle{\bf f}_1,{\bf
a}_2\rangle}\langle{\bf f}_1,{\bf a}_{n+1}\rangle\right).
\end{eqnarray*}
By the assumptions,
\begin{eqnarray}
&&C_{\mathcal{Z}_{\delta_m}^{k_m}}(\kappa_m^2 r_m,r_m;0,P_{1})\nonumber\\
&\ge& \sum_{\substack{j=1\\j\neq
n+1}}^{2n+2}C_{\mathcal{Z}_{\delta_m}^{k_m}}^{1)}(\kappa_m^2
r_m,r_m;H_j,f_l)+C_{\mathcal{Z}_{\delta_m}^{k_m}}(\kappa_m^2
r_m,r_m;H_1,f_l)\nonumber\\
&&+C_{\mathcal{Z}_{\delta_m}^{k_m}}^{n+1)}(\kappa_m^2
r_m,r_m;H_{n+1},f_l)\nonumber\\
&=&\sum_{\substack{j=1\\j\neq
n+1}}^{2n+2}C_{\mathcal{Z}_{\delta_m}^{k_m}}^{1)}(\kappa_m^2
r_m,r_m;H_j,f_l)+C_{\mathcal{Z}_{\delta_m}^{k_m}}(\kappa_m^2
r_m,r_m;H_1,f_l)\nonumber\\
&&+C_{\mathcal{Z}_{\delta_m}^{k_m}}(\kappa_m^2
r_m,r_m;H_{n+1},f_l)+o(1)\frac{T(r_m)}{r_m^{\omega_m}}.\label{eqnN12}
\end{eqnarray}
By (\ref{eqnN8}), (\ref{eqnN9}) and (\ref{eqnN12}), we have
\begin{eqnarray*}
&&\sum_{\substack{j=1\\j\neq 1,
n+1}}^{2n+2}C_{\mathcal{Z}_{\delta_m}^{k_m}}^{1)}(\kappa_m^2
r_m,r_m;H_j,f_l)+\sum_{j=1,n+1}C_{\mathcal{Z}_{\delta_m}^{k_m}}(\kappa_m^2
r_m,r_m;H_j,f_l)\\
&\ge &\sum_{\substack{j=1\\j\neq
n+1}}^{2n+2}C_{\mathcal{Z}_{\delta_m}^{k_m}}^{1)}(\kappa_m^2
r_m,r_m;H_j,f_l)+\sum_{j=1,n+1}C_{\mathcal{Z}_{\delta_m}^{k_m}}(\kappa_m^2
r_m,r_m;H_j,f_l)+o(1)\frac{T(r_m)}{r_m^{\omega_m}},
\end{eqnarray*}
i.e.,
\begin{eqnarray}
C_{\mathcal{Z}_{\delta_m}^{k_m}}^{1)}(\kappa_m^2
r_m,r_m;H_1,f_l)=o(1)\frac{T(r_m)}{r_m^{\omega_m}},\ l=1,2.\label{eqnN13}
\end{eqnarray}
From (\ref{eqnN11}) and (\ref{eqnN13}), it follows that
\begin{eqnarray}
S_{\mathcal{Z}_{\delta_m}^{k_m}}(r_m,f_{l})=C_{\mathcal{Z}_{\delta_m}^{k_m}}^{n)}(\kappa_m^2
r_m,r_m;H_{n+1},f_l)+o(1)\frac{T(r_m)}{r_m^{\omega_m}},\ l=1,2.\label{eqnN14}
\end{eqnarray}

Set $Q_i=\langle{\bf f}_1,{\bf a}_i\rangle\langle{\bf f}_2,{\bf a}_{n+1}\rangle-\langle{\bf f}_1,{\bf a}_{n+1}\rangle\langle{\bf f}_2,{\bf a}_i\rangle$ and
put $\mathcal{Q}=\{1\le i\le 2n+2:Q_i\not\equiv 0\}$. Suppose that $\sharp \mathcal{Q}\ge n+2$. Take $n+2$ elements of $\mathcal{Q}$, written as $i_j(1\le j\le n+2)$. By $Q_{i_j}\not\equiv 0$, we have
\begin{eqnarray*}
&&\sum_{l'=1}^2S_{\mathcal{Z}_{\delta_m}^{k_m}}(r_m,f_{l'})+O(1)\ge C_{\mathcal{Z}_{\delta_m}^{k_m}}(\kappa_m^2
r_m,r_m;0,Q_{i_j})\\
&\ge&C_{\mathcal{Z}_{\delta_m}^{k_m}}^{n+1)}(\kappa_m^2
r_m,r_m;H_{i_j},f_l)+C_{\mathcal{Z}_{\delta_m}^{k_m}}^{n+1)}(\kappa_m^2
r_m,r_m;H_{n+1},f_l)\\
&&+\sum_{\substack{j=1\\j\neq i_j,
n+1}}^{2n+2}C_{\mathcal{Z}_{\delta_m}^{k_m}}^{1)}(\kappa_m^2
r_m,r_m;H_j,f_l)\\
&=&C_{\mathcal{Z}_{\delta_m}^{k_m}}^{n)}(\kappa_m^2
r_m,r_m;H_{i_j},f_l)+C_{\mathcal{Z}_{\delta_m}^{k_m}}^{n)}(\kappa_m^2
r_m,r_m;H_{n+1},f_l)\ \
(\text{by\ (\ref{eqnN9})})\\
&&+\sum_{\substack{j=1\\j\neq i_j,
n+1}}^{2n+2}C_{\mathcal{Z}_{\delta_m}^{k_m}}^{1)}(\kappa_m^2
r_m,r_m;H_j,f_l)+o(1)\frac{T(r_m)}{r_m^{\omega_m}}\\
&=&\frac{n-1}{n}C_{\mathcal{Z}_{\delta_m}^{k_m}}^{n)}(\kappa_m^2
r_m,r_m;H_{n+1},f_l)+\frac{n-1}{n}C_{\mathcal{Z}_{\delta_m}^{k_m}}^{n)}(\kappa_m^2
r_m,r_m;H_{i_j},f_l)\\
&&+\frac{1}{n}\sum_{j=1}^{2n+2}C_{\mathcal{Z}_{\delta_m}^{k_m}}^{n)}(\kappa_m^2
r_m,r_m;H_j,f_l)+o(1)\frac{T(r_m)}{r_m^{\omega_m}}\ \
(\text{by\ (\ref{eqnN6})})\\
&=&\frac{n-1}{n}S_{\mathcal{Z}_{\delta_m}^{k_m}}(r_m,f_l)+\frac{n-1}{n}C_{\mathcal{Z}_{\delta_m}^{k_m}}^{n)}(\kappa_m^2
r_m,r_m;H_{i_j},f_l)\ \
(\text{by\ (\ref{eqnN14})})\\
&&+\frac{n+1}{n}S_{\mathcal{Z}_{\delta_m}^{k_m}}(r_m,f_l)+o(1)\frac{T(r_m)}{r_m^{\omega_m}}\ \
(\text{by\ (\ref{eqnN7})})\\
&=&2S_{\mathcal{Z}_{\delta_m}^{k_m}}(r_m,f_l)+\frac{n-1}{n}C_{\mathcal{Z}_{\delta_m}^{k_m}}^{n)}(\kappa_m^2
r_m,r_m;H_{i_j},f_l)+o(1)\frac{T(r_m)}{r_m^{\omega_m}}\\
&=&\sum_{l=1}^2S_{\mathcal{Z}_{\delta_m}^{k_m}}(r_m,f_l)+\frac{n-1}{n}C_{\mathcal{Z}_{\delta_m}^{k_m}}^{n)}(\kappa_m^2
r_m,r_m;H_{i_j},f_l)+o(1)\frac{T(r_m)}{r_m^{\omega_m}}.\ \
(\text{by\ (\ref{eqnN10})})
\end{eqnarray*}
Thus,
\begin{eqnarray}
C_{\mathcal{Z}_{\delta_m}^{k_m}}^{n)}(\kappa_m^2
r_m,r_m;H_{i_j},f_l)=o(1)\frac{T(r_m)}{r_m^{\omega_m}},\ l=1,2.\label{eqnN15}
\end{eqnarray}
Using Theorem \ref{thm2.1} for $\{H_{i_j}\}_{j=1}^{n+2}$, we have, by (\ref{eqnN1}), (\ref{eqnN2}), (\ref{eqnN3}) and (\ref{eqnN15}),
\begin{eqnarray*}
\sum_{l=1}^{2}S_{\mathcal{Z}_{\delta_m}^{k_m}}(r_m,f_l)\le\sum_{l=1}^{2}\sum_{j=1}^{n+2}C_{\mathcal{Z}_{\delta_m}^{k_m}}^{n)}(\kappa_m^2
r_m,r_m;H_{i_j},f_l)+o(1)\frac{T(r_m)}{r_m^{\omega_m}}
=o(1)\frac{T(r_m)}{r_m^{\omega_m}}.
\end{eqnarray*}
This implies that $\frac{\sqrt{2}}{10}\frac{T(r_m)}{r_m^{\omega_m}}\le o(1)\frac{T(r_m)}{r_m^{\omega_m}}$ ($m\rightarrow \infty$) by (\ref{eqnN4}), a contradiction is derived. Hence $\sharp \mathcal{Q}\le n+1$, which is contradict to $f\not\equiv g$. \qed

\begin{re}
The proof of Theorem \ref{thm1.7} is similar to that of Theorem \ref{thm1.5} if one note that
$$
T(e^{r_m})\geq \frac{r_m}{\log 2r_m}T(2r_m)+(1-\frac{r_m}{\log 2r_m})T(1),
$$
i.e., $T(2r_m)=o(T(e^{r_m}))$.
\end{re}

\

{\sl Proof of Theorem \ref{thm1.10}.} \ Suppose $f\not\equiv g$ and set $f_1:=f$ and
$f_2:=g$. Similar to the proof of Theorem \ref{thm1.2}, we take $r'_m$ in $[r_m/2,r_m]$ outside the exceptional set
in second main theorems used blow which only has finite linear measure. Now we take $r=r'_m$.
Then
we have
$$\sum_{j=1,2}(N^{1)}(r^\sigma,a_j,f_l)-N^{1)}(2r,a_j,f_l))+\sum_{j=
3,4}(N^{2)}(r^\sigma,a_j,f_l)-N^{2)}(2r,a_j,f_l))$$$$\le T(r^\sigma)+O(1)$$
and
$$\sum_{j=3,4}(N^{1)}(r^\sigma,a_j,f_l)-N^{1)}(2r,a_j,f_l))+\sum_{j=
1,2}(N^{2)}(r^\sigma,a_j,f_l)-N^{2)}(2r,a_j,f_l))$$$$\le T(r^\sigma)+O(1).$$

Summing-up the above inequalities, we obtain
$$
2\sum_{j=1}^4(N^{1)}(r^\sigma,a_j,f_l)-N^{1)}(2r,a_j,f_l))
$$$$+\sum_{j=1}^4((N^{2)}(r^\sigma,a_j,f_l)-N^{1)}(r^\sigma,a_j,f_l))-(N^{2)}(2r,a_j,f_l)-N^{1)}(2r,a_j,f_l)))$$
$$\le 2T(r^\sigma)+O(1).$$
We have, for $l=1,2$,
\begin{eqnarray*}
&&2\sum_{j=1}^{4}N^{1)}(r^\sigma,a_j,f_l)
+\sum_{j=1}^{4}(N_{>1}^{1)}(r^\sigma,a_j,f_l)-N_{>1}^{1)}(2r,a_j,f_l))\\
&\le& 2T(r^\sigma)+2\sum_{j=1}^{4}N^{1)}(2r,a_j,f_l)+O(1),
\end{eqnarray*}
where $N_{>1}^{1)}(r^\sigma,a_j,f_l)$ and $N_{>1}^{1)}(2r,a_j,f_l)$ are
the counting functions in which we only consider the zeros of
$\langle{\bf f}_l,{\bf a}_j\rangle$ with multiplicity $>1$.

Using the second main theorem, we have
\begin{eqnarray*}
&&4T(r^\sigma)+\sum_{j=1}^{4}\sum_{l=1}^2(N_{>1}^{1)}(r^\sigma,a_j,f_l)-N_{>1}^{1)}(2r,a_j,f_l))\\
&\le&4T(r^\sigma)+8T(2r)+O(\log^+rT(r^\sigma)).
\end{eqnarray*}
It implies that
\begin{eqnarray}
\sum_{j=1}^{4}((N_{>1}^{1)}(r^\sigma,a_j,f_1)-N_{>1}^{1)}(2r,a_j,f_1))+(N_{>1}^{1)}(r^\sigma,a_j,f_2)-N_{>1}^{1)}(2r,a_j,f_2)))\nonumber\\
\le 8T(2r)+O(\log^+rT(r^\sigma)).\label{eqn00n-10}
\end{eqnarray}

By using the same notations and repeating the same argument in the proof of Theorem \ref{thm1.5}, we have that if $\tau\geq 2$, then the holomorphic curve
$h:{\mathbb{C}}\rightarrow {\mathcal{P}}^{\tau-1}({\mathbb{C}})$ with
a reduced representation
$$
{\bf h}=\left(b_1\tilde hh_{I''_1},b_2\tilde hh_{I''_2},..., b_{\tau}\tilde
hh_{I''_{\tau}}\right)
$$
is linearly non-degenerate over ${\mathbb{C}}$.

Using Cartan's second main theorem, we hold that
\begin{eqnarray}
T(r^\sigma,h)
&\le&(\tau-1)\sum_{i=1}^{\tau}N^{1)}(
r^\sigma,0,\tilde{h}h_{I''_{i}})+(\tau-1)N^{1)}(
r^\sigma,0,\tilde
{h}\frac{h_{I'_0} }{h_{I'}})\nonumber\\
&+&o(T(r^\sigma,h)),\label{eqn00n-13}
\end{eqnarray}
where $T(r^\sigma,h)\le O(T(r^\sigma))$.

On the other hand,
\begin{eqnarray*}
N^{1)}(r^\sigma,0,\tilde{h}h_{I''_{i}})\le N^{1)}(r^\sigma,0,h_{I_i})+N^{1)}(r^\sigma,0,1/h_{I^c_i})
\end{eqnarray*}
and
\begin{eqnarray*}
N^{1)}(r^\sigma,0,\tilde
{h}\frac{h_{I'_0} }{h_{I'}})\le N^{1)}(r^\sigma,0,h_{I_0})+N^{1)}(r^\sigma,0,1/h_{I^c_0}).
\end{eqnarray*}

Now we consider the zeros and poles of $h_i$. If $z$ is a zero(a pole) of $h_i$ in $A(\{r_m\},\sigma)$, then
$\nu_{\langle{\bf f}_1,{\bf a}_i\rangle}(z)\neq \nu_{\langle{\bf f}_2,{\bf a}_i\rangle}(z)$. Hence
$\nu_{\langle{\bf f}_l,{\bf a}_i\rangle}(z)>1$.

By (\ref{eqn00n-10}), we have
\begin{eqnarray*}
&&N^{1)}(r^\sigma,0,\tilde{h}h_{I''_{i}})\\
&\le& \sum_{j=1}^{4}((N_{>1}^{1)}(r^\sigma,a_j,f_1)-N_{>1}^{1)}(2r,a_j,f_1))+(N_{>1}^{1)}(r^\sigma,a_j,f_2)-N_{>1}^{1)}(2r,a_j,f_2)))\\
&&+\sum_{j=1}^{4}(N^{1)}(2r,a_j,f_1)-N^{1)}(2r,a_j,f_2))\\
&\le& 12T(2r)+o(T(r^\sigma))
\end{eqnarray*}
and
\begin{eqnarray*}
N^{1)}(r^\sigma,0,\tilde
{h}\frac{h_{I'_0} }{h_{I'}})\le 12T(2r)+o(T(r^\sigma)).
\end{eqnarray*}
Combining with the above two inequalities and (\ref{eqn00n-13}) yields
\begin{eqnarray}
T(r^\sigma,h)\le 12(\tau^2-1)T(2r)+o(T(r^\sigma)).\label{eqn00n-14}
\end{eqnarray}

Since $ f_1=f_2$ on $A(\{r_m\},\sigma)\cap\bigcup_{j=1}^{4}
f^{-1}(a_j)$. That is to say, $\frac{h_I}{h_J}=1$ on the set
$A(\{r_m\},\sigma)\cap\bigcup_{j\in((I\cup J)\backslash(I\cap
J))^c}f^{-1}(a_j)$. Hence,
\begin{eqnarray*}
3T(r^\sigma,h)&\ge & N^{1)}(r^\sigma,0,\frac{h_{I_1}}{h_{I_2}}-1)+N^{1)}(r^\sigma,0,\frac{h_{I_2}}{h_{I_0}}-1)+N^{1)}(r^\sigma,0,\frac{h_{I_0}}{h_{I_1}}-1)+O(1)\nonumber\\
&\ge&\sum_{j=1}^{4}N^{1)}(r^\sigma,a_j,f_{l})-\sum_{j=1}^{4}N^{1)}(2r,a_j,f_{l})+O(1),\ \ l=1,2.
\end{eqnarray*}
Using the second main theorem again, we have
\begin{eqnarray*}
3T(r^\sigma,h)+4T(2r,f_l)&\ge&3T(r^\sigma,h)+\sum_{j=1}^{4}N^{1)}(2r,a_j,f_{l})\\
&\ge&2T(r^\sigma,f_l)+o(T(r^\sigma,f_l))\ \ l=1,2.
\end{eqnarray*}
Hence,
\begin{eqnarray}
T(r^\sigma,h)+\frac{2}{3}T(2r)\ge \frac{1}{3}T(r^\sigma)+o(T(r^\sigma)).\label{eqn00n-15}
\end{eqnarray}
In view of (\ref{eqn00n-14}) and (\ref{eqn00n-15}),
$$
\frac{1}{3}T(r^\sigma)+o(T(r^\sigma))\le (12(\tau^2-1)+\frac{2}{3})T(2r)+o(T(r^\sigma)).
$$
By
$$
T(r^\sigma)\geq \frac{\sigma\log r}{\log 2r}
T(2r)+(1-\sigma)T(1),$$
we have $\frac{\sigma}{3}\le 12(\tau^2-1)+\frac{2}{3}$. We note that $\tau\le 5$, which is contradict to $\sigma>866$. Hence $\tau=1$. Then the proof of Theorem \ref{thm1.10} can be completed. \qed

 \vskip 0.5cm

%{\bf Acknowledgements.} %I would like to express my gratitude to Alex
%Eremenko for his valuable comments and for him to tell me Fedorov
%and Grishin's paper \cite{FedorovGrishin} and to Ru Min for his
%valuable suggestion in improving this paper.


\begin{thebibliography}{99}

\bibitem{CaoCao} H. Z. Cao and T. B. Cao,  Uniqueness problem for meromorphic mappings in several complex variables with few hyperplanes, Acta Math. Sin. (Engl. Ser.), 31(8) (2015), 1327-1338

\bibitem{Cartan}H. Cartan, Sur les z\'eros des combinaisons
lin\'eaires de $p$ fonctions holomorphes donn\'ees, Mathematica
Cluj, 7(1933), 5-31

\bibitem{ChenRu}Z. Chen and M. Ru, A uniqueness theorem for moving
targets with truncated multiplicities, Houston J. Math., 32(2006),
589-601

\bibitem{ChenYan}Z. Chen and Q. Yan, Uniqueness theorem of
meromorphic mappings into $\mathcal{P}^n(\mathbb{C})$ sharing $2N+3$
hyperplanes regardless of multiplicities, Internat. J. Math., 20(6) (2009), 717-726

%\bibitem{CHN}T. Ganelius, W. Hayman and D. J. Newman, Lectures on
%approximation and value distribution; value distribution and
%exceptional sets, S\'eminaire de math\'ematiques sup\'erieures (Les
%presses de l'Universit\'e de Montr\'eal, 1982), 79-147

\bibitem{DethloffTan} G. Dethloff and V. T. Tran, Uniqueness theorems for meromorphic mappings with few hyperplanes, Bull. Sci. Math., 133(5) (2009), 501-514

\bibitem{DethloffTan1} G. Dethloff and V. T. Tran, A uniqueness theorem for meromorphic maps with moving hypersurfaces, Publ. Math. Debrecen, 78(2) (2011), 347-357

\bibitem{DethloffSyTan} G. Dethloff, D. Q. Sy and T. V. Tran, A uniqueness theorem for meromorphic mappings with two families of hyperplanes, Proc. Amer. Math. Soc., 140(1) (2012),  189-197

%\bibitem{DrasinHayman}D. Drasin and W. Hayman, Value distribution of functions meormorphic in an angle,
%Proc. London Math. Soc., (3), 48(1984), 319-340

\bibitem{ThaiQuang} D. T. Do and D. Q. Si, Uniqueness problem with truncated
multiplicities in several complex variables, Internat. J. Math.,
17(10) (2006), 1223-1257

\bibitem{DulockRu} M. Dulock and M. Ru, Uniqueness of holomorphic
curves into abelian varieties, Trans. Amer. Math. Soc., 363(2011),
131-142

%\bibitem{Eremenko} A. Eremenko, Julia directions for holomorphic
%curves, Preprint in A. Eremenko's Home pape.

%\bibitem{Eremenko1} A. Eremenko, Value distribution and potential
%theory, Proceedings of the ICM, vol. 2, 681-690, Higher Education
%Press, Beijing, 2002.

%\bibitem{EremenkoSodin} A. Eremenko and M. L. Sodin, The value
%distribution of meromorphic functions and meromorphic curves from
%the point of view of potential theory, St. Petersburg Math. J.,
%3(1992), 109-136.

%\bibitem{FedorovGrishin}M. A. Fedorov and A. F. Grishin, Some questionss of
%Nevanlinna theory for the complex half plane, Math. Physics,
%Analysis and Geometry, vol. 1(1998), 223-271.

\bibitem{Fujimoto} H. Fujimoto, The uniqueness problem of
meromorphic maps into the complex projective space, Nagoya Math. J.,
58(1975), 1-23

\bibitem{Fujimoto2} H. Fujimoto, A uniqueness theorem for
algebraically non-degenerate meromorphic maps into
$\mathcal{P}^n(\mathbb{C})$, Nagoya Math. J., 64(1976), 117-147

\bibitem{Fujimoto3} H. Fujimoto, Remarks on the uniqueness problem of
meromorphic maps into $\mathcal{P}^n(\mathbb{C})$,III, Nagoya Math.
J., 75(1979), 71-85

\bibitem{Fujimoto4} H. Fujimoto, Uniqueness problem with truncated
multiplicities in value distribution theory, Nagoya Math. J.,
152(1998), 131-152

\bibitem{Fujimoto1} H. Fujimoto, Uniqueness problem with truncated
multiplicities in value distribution theory,II, Nagoya Math. J.,
155(1999), 161-188

%\bibitem{Gleizer} E. V. Gleizer, Meromorphic functions with zeros
%and poles in small angles, Sib. Mat. Zh., (4)26(1985), 22-37;
%II,Sib. Mat. Zh., (2)31(1990), 9-20

%\bibitem{Goldberg} A. A. Goldberg, Nevanlinna's lemma on the
%logarithmic derivative of meromorphic function, Mat.Zametki,
%(4)17(1975), 525-529

%\bibitem{GoldbergeOstrovskii} A. A. Goldberg and I. V. Ostrovskii, \textit{Value Distribution of Meromorphic Functions},
%Translations of Mathematical Monographs, vol. 236, 2008.

%\bibitem{GundersenHayman} G. Gundersen and W. Hayman, The strength
%of  Cartan's version of Nevanlinna theory, Bull. London Math. Soc.,
%36(2004), 433-454.

%\bibitem{Hayman} W. K. Hayman, \textit{Meromorphic Functions}, Clarendon Press, Oxford, 1964.

%\bibitem{Hayman1} W. K. Hayman, Waring's problem f\"{u}r analytische
%Funktionen, Bayer. Akad. Wiss. Math. Natur. Kl. Sitzungsber. 1984
%(Bayer. Akad. Wiss., Mnich, 1985) 1-13

%\bibitem{HaymanYang} W. Hayman and L. Yang, Growth and values of
%functiona regular in an angle, Proc. London Math. Soc., (3)44(1982),
%193-214

\bibitem{HaLeSi} H. G. Ha, N. Q. Le and D. Q. Si, Uniqueness theorems for meromorphic mappings sharing few hyperplanes, J. Math. Anal. Appl., 393(2) (2012),  445-456

\bibitem{Ji} S. Ji, Uniqueness problem without multiplicities in
value distribution theory, Pacific J. Math., 135(1988), 323-348

%\bibitem{LuGu}Q. Lu and Y. X. Gu, On Nevanlinna directions of
%algebroid functions, Acta Math. Scientia, (2){\bf 25(B)}(2005),
%367-375

%\bibitem{Niino} K. Niino, General defect relations of holomorphic
%curves, Trans. Amer. Math. Soc., vol. 289, no.1(1985), 99-113.

%\bibitem{Petrenko}V. P. Petrenko, Growth and distribution of values
%of algebroid functions, Math. Zametki, vol.26, no.4(1979), 513-522

\bibitem{Ru} M. Ru, Nevanlinna theory and its relation to
diophantine approximation, World Scientific, 2001.

\bibitem{Ru1} M. Ru, A uniqueness theorem with moving targets
without counting multiplicity, Proc. Amer. Math. Soc., 129(9) (2001), 2701-2707

%\bibitem{RuStoll} M. Ru and W. Stoll, The second main theorem for
%moving targets, J. Geom. Anal., 1(1991), 99-138

%\bibitem{S} H. Selbreg, \textit{Algebroide Funktionen und Umkehrfunktionen Abelscher
%Integrale},  Avh. Norske Vid. Akad. Oslo \textbf{8}(1934), 1-72.

\bibitem{Si2} D. Q. Si, Unicity of meromorphic mappings sharing few
hyperplnes, Annalis Polonici Math., (2011), 102(3), 255-270

\bibitem{Si} D. Q. Si, Some extensions of the four values theorem of Nevanlinna-Gundersen, Kodai Math. J., 36(3) (2013), 579-595

\bibitem{Si1} D. Q. Si, Degeneracy and finiteness theorems for meromorphic mappings in several complex variables, Chin. Ann. Math. Ser. B, 40(2) (2019), 251-272

\bibitem{SiLe} D. Q. Si, N. G. Le, Two meromorphic mappings sharing $2n+2$ hyperplanes regardless of multiplicity, J. Math. Anal. Appl., 410(2) (2014),  771-782

\bibitem{Smiley} L. Smiley, Geometric conditions for unicity of
holomorphic curves, Contemp. Math., 25(1983), 149-154

\bibitem{Stoll} W. Stoll, On the propagation of dependences, Pacific
J. Math., 139(1989), 311-337

\bibitem{Tan} V. T. Tran, A degeneracy theorem for meromorphic mappings with few hyperplanes and low truncation level of multiplicities, Publ. Math. Debrecen, 74(2009), 279-292

\bibitem{TanTruong} V. T. Tran and V. T. Vu, Three meromorphic
mappings sharing some common hyperplanes, J. Math. Anal. Appl.,
348(2008), 562-570



\bibitem{YanChen} Q. Yan and Z. Chen, A degeneracy theorem for meromorphic mappings with truncated multiplicities, Acta Math. Sci. Ser. B Engl. Ed., 31(2) (2011), 549-560

\bibitem{Yang} L. Yang, Value distribution of meromorphic functions,
Springer, 1993

%\bibitem{Tsuji} M. Tsuji, Potential theory in modern function
%theory, Maruzen Co. LTD, Tokyo, 1959.

%\bibitem{Tu} Z. H. Tu, On the Julia directions of the value
%distribution of holomorphic curves in $\mathcal{P}(\mathbb{C}^n)$,
%Kodai Math. J., 19(1996), 1-6

%\bibitem {U} E. Ullrich, \textit{\"{U}ber den Einfluss der verzweigtheit einer Algebroide auf ihre
%Wertverteilung}, J.reine ang. Math.  \textbf{169}(1931), 198-220.

%\bibitem {V1} G. Valiron, \textit{Sur la deriv\'{e}e des fonctions
%alg\'{e}bro\"{\i}des}, Bull. Sci. Math. \textbf{59}(1931), 17-39.

%\bibitem{Wu}Z. J. Wu, On $T$ direction of algebroid functions, J.
%Math. Kyoto Univ., (4){\bf 47}(2007), 767-779

%\bibitem{Xuan} Z. X. Xuan, On the existence of $T$ direction of
%algebroid functions: a problem of J. H. Zheng, J. Math. Anal. Appl.,
%(1){\bf 341}(2008), 540-547

%\bibitem{Zhangqd} Q. D. Zhang, $T$ directions and Borel directions
%of meromorphic functions with finite positive order, Acta Math.
%Sinica, (2)50(2007), 413-420

%\bibitem{ZhangXL}X. L. Zhang, A fundamental inequality for
%meromorphic functions in an angular domain and its application, Acta
%Math. Sinica, New Series, (3)10(1994), 308-314.

%\bibitem{Zheng1}J. H. Zheng, On transcendental meromorphic functions
%with radially distributed values, Science in China, Series A. Math.,
%(3), 47(2004), 401-416

\bibitem{Zheng} J. H. Zheng, Value distribution of meromorphic
functions, Springer-Verlag Berlin Heidelberg, 2010.

\bibitem{Zheng1} J. H. Zheng, Value distribution of holomorphic
curves on an angular domain, Michigan Math. J., 64(2015), 849-879

\bibitem{Zheng2} J. H. Zheng, Local value distribution of
holomorphic curves, preprint.


\end{thebibliography}
\end{document}